%% file: DiscreteApprox2.tex
\documentclass[bj]{imsart}

\startlocaldefs
\include{preamble2}
\newcounter{thm}

\theoremstyle{plain}
\newtheorem{theorem}[thm]{Theorem}
\numberwithin{thm}{section}
\newtheorem{defn}[thm]{Definition}
\newtheorem{lemma}[thm]{Lemma}
\newtheorem{prop}[thm]{Proposition}
\newtheorem{cor}[thm]{Corollary}
\newtheorem{P-C}[thm]{Theorem}
\newtheorem{H-L}[thm]{Theorem}

\theoremstyle{definition}
\newtheorem{remark}[thm]{Remark}
\endlocaldefs

\begin{document}

\begin{frontmatter}

% "Title of the Paper"
\title{Convergence of jump processes with stochastic intensity to Brownian motion with inert drift}

\runtitle{Approximating Brownian motion with inert drift by jump processes}

\begin{aug}
% indicate corresponding author with \corref{}
% \author{\fnms{John} \snm{Smith}\thanksref{a}\corref{}\ead[label=e1]{smith@foo.com}\ead[label=e2,url]{www.foo.com}}
% \address[a]{\printead{e1};\printead{e2}}

\author{\fnms{Clayton} \snm{Barnes}\thanksref{a}\ead[label=e1]{cbarnes@campus.technion.ac.il}}
%\and
%\author{\fnms{???} \snm{???}\thanksref{b}\ead[label=e2]{???}}
\address[a]{\printead{e1}}
%\address[b]{\printead{e2}}

\runauthor{Clayton Barnes}

\affiliation{Technion-Israel's Institute of Technology}

\end{aug}

\begin{abstract}
Consider a random walker on the nonnegative lattice, moving in continuous time, whose positive transition intensity is proportional to the time the walker spends at the origin. In this way, the walker is a jump 
process with a stochastic and adapted jump intensity. We show that, upon Brownian scaling, the sequence of such processes converges to Brownian motion with inert drift (BMID). BMID was introduced by Frank Knight in 2001 and generalized by White in 2007. This confirms a conjecture of Burdzy and White in 
2008 in the one-dimensional setting.
\end{abstract}


% history:
% \received{\smonth{1} \syear{0000}}

%\tableofcontents

\end{frontmatter}

% \maketitle
\section{Introduction}

Brownian motion with inert drift (BMID) is a process $X$ that satisfies the SDE
\begin{align}\label{eq:SDE_BMID}
dX = dB + dL + KLdt,
\end{align}
where $K \geq 0$ is a constant and $L$ is the local time of $X$ at zero. This process 
behaves as a Brownian motion away from the origin but has a drift proportional to its local time at zero. Note that such a process is not Markovian because its drift depends on the past history. See Figure \ref{ch3:fig1} for sample path comparisons between reflected Brownian motion and BMID. BMID can be constructed path-by-path from a standard Brownian motion via the employment of a Skorohod map. This is discussed in more detail in Section \ref{section:BMID}. We consider continuous time processes $(X_n, V_n)$ on $2^{-n}\N \times \R$ such that for $K \geq 0$,
\begin{enumerate}[label = (\roman*)]
\item $V_n(t) := K2^n\cdot Leb(0 < s < t : X_n(s) = 0)$ is the scaled time $X_n$ spends at the origin.\\
\item $X_n$ is a jump process with positive jump intensity $2^{2n} + 2^nV_n(t)$ and downward jump intensity 
$2^{2n}$, modified appropriately so $X_n$ does not transition below zero.
\end{enumerate}
The existence of such a process and its rigorous definition is presented in Section \ref{ch3:section:MarkovProcessesMemory}. Intuitively, $X_n$ is a random walker on the lattice $2^{-n}\N$ whose transition rates depend linearly on the amount of time the walker spends at zero.
% The upward jump rate of $X_n$ increases as it spends more time at zero.
In other words, the positive jump rate of $X_n$ increases each time $X_n$ reaches zero.
We show that as the lattice size shrinks to zero, i.e.\ as $n \to \infty$, $(X_n, V_n)$ converges in distribution to $(X, V)$, where $X$ is BMID and $V = KL$ is its velocity. See Theorem \ref{ch3:thmDiscreteConvergence} and Corollary \ref{th:cor_discrete_conv} for precise statements. By setting $K = 0$ we recover the classical result that random walk on 
the nonnegative lattice converges to reflected Brownian motion.

\begin{figure}
\includegraphics[scale = .6]{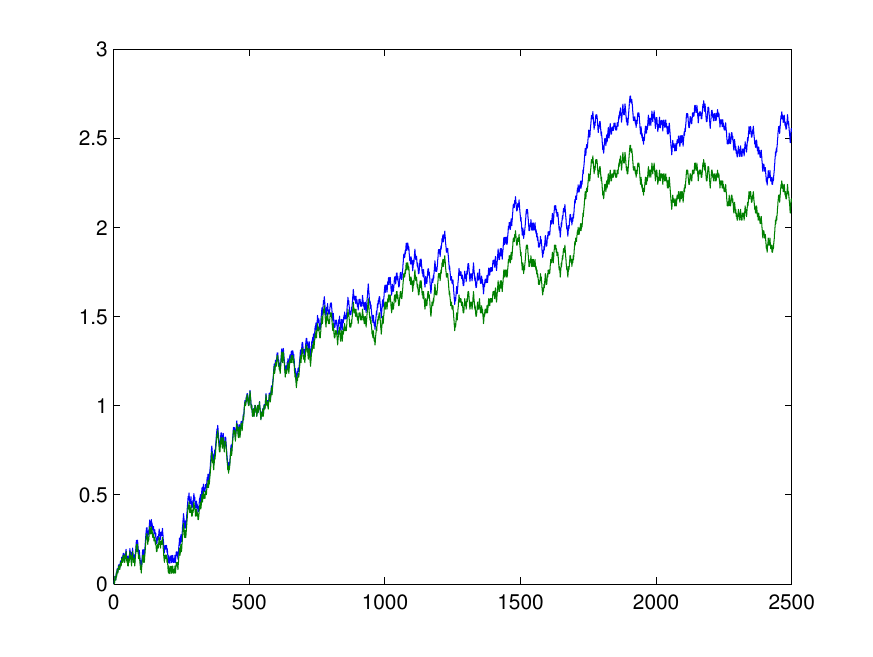}\label{ch3:fig1}
\caption{The figure shows a path of reflected Brownian motion and a path of BMID, each coming from the same 
Brownian sample path. This drift of BMID becomes evident as its domination over reflected Brownian motion grows with time. The drift of BMID comes from its contact with zero, which, in this figure, occurs in the beginning of the process.}
\end{figure}

\subsection{Outline} In Section 2 we introduce BMID and its construction using Skorohod maps. We also give an equivalent formulation of the process $(X, V)$. In Section 3 we introduce the necessary 
background on jump process with stochastic intensity and introduce the setting used by Burdzy and White in \cite{Burdzy:White}.
 % and restate a key theorem of theirs as Theorem \ref{th2}. 
Section 4 contains the statement and proof of the main results, Theorem \ref{ch3:thmDiscreteConvergence} and Corollary \ref{th:cor_discrete_conv}. We conclude by briefly discussing BMID in a multidimensional setting in Section \ref{sec:mutli} .

\subsection{Background}The study of BMID began in 2001 when Knight \cite{Knight2001} described a Brownian particle reflecting above a particle with Newtonian dynamics. 
% $Z$ is a diffusion reflecting inside a sufficiently smooth domain $D \subset \R^n$, and $V$ is its drift.
% This drift is the inward normal integrated against the local time $Z$ spends on $\partial D$. 
% % That is,
% % \begin{align}
% % \begin{split}
% % Z(t) &= B(t) + \int_0^t\eta(Z(s))\,\mathrm{d}L(s) + \int_0^tV(s)\, \mathrm{d}s, \label{ch3:eq1}\\
% % V(t) &= V_0 + \int_0^t\eta(Z(s))\, \mathrm{d}L(s),
% % \end{split}
% % \end{align}
% % \noindent
% % where $\eta(x)$ is the inward unit normal for $x \in \partial D$ and $t \to L(t)$ is a nondecreasing
% % continuous function flat off of $\partial D.$ By this we mean $L$ increases only on $Z^{-1}(\partial D).$
% The authors show $(Z, V)$ has a stationary distribution of $\mu \times \gamma$, where $\mu$ is the uniform
% distribution on $D$ and $\gamma$ is the Gaussian distribution on $\R^n$. This is interesting in part because the stationary distribution of the drift is always Gaussian and does not depend on $D$, and also because the stationary distribution is always a product form.
% When $Z$ is one dimensional, and $D = [0, \infty)$, the process $Z$ is one dimensional reflected BMID which is the process introduced by Knight. 
This two-particle system of Knight is equivalent in some sense to BMID in that the gap between the
Brownian particle and the Newtonian particle is BMID. See Section \ref{section:BMID}. For more background on BMID see \cite{white2007}, where White constructs a multidimensional analog to BMID; \cite{bass2010stationary}, where Bass, Burdzy, Chen and Hairer study the stationary distribution; \cite{Barnes}, where Barnes describes the hydrodynamic behavior of systems of Brownian motions with inert drift.

Burdzy and White studied similar processes from a discrete state point of view \cite{Burdzy:White}. They 
consider a pair of processes $(X, L)$ with state space $\mathscr{L}\times \R^d$, where $\mathscr{L}$ is a finite set, and where the transition rate of $X$ depends on $L$, the scaled time $X$
 has spent on previous states. See subsection \ref{Ch3:section:ClassC} for definitions.
% We call this class of processes $\mc{C}$ and describe it in more detail in subsection \ref{Ch3:section:ClassC}. 
The authors 
find necessary and sufficient conditions for such a process $(X, L)$ to have stationary distribution $\mu \times \gamma$, where $\mu$ is uniform on space and $\gamma$ is Gaussian. Burdzy and White make many conjectures involving 
approximating BMID, and its variants, and suggest the results of Bass, Burdzy, Chen, and Hairer \cite{bass2010stationary} concerning a multidimensional analog of BMID stem from a discrete approximation scheme where the continuous process of BMID is a limit of these processes  whose values take place in discretized space. The main result of this article confirms the discrete approximation scheme converges to the continuous model in the one 
dimensional setting.

BMID is just one example where a process with memory has a Gaussian stationary distribution. Gauthier \cite{gauthier2018central} studies diffusions whose drift is also dependent on the process history through a linear combination of sine and cosine functions. He shows the average displacement across time obtains a Gaussian stationary distribution as time approaches infinity. In \cite{barnes2019billiards}, Barnes, Burdzy, and Gauthier use this discrete approximation scheme, taking limits of Markov processes in the same class considered here, to demonstrate billiards with certain Markovian reflection laws have $\mu \times \gamma$ as the stationary measure for space and velocity,  where as above $\mu$ is the uniform stationary measure in the spatial component and $\gamma$ the Gaussian stationary measure in the velocity component.

In Section \ref{sec:mutli} we briefly discuss a multidimensional analog of BMID that inspired conjectures of Burdzy and White.
% The second contribution of this chapter concerns a pair of reflected Brownian motions $(X_1, X_2)$ separated by a Newtonian/inert particle $Y$ whose drift is proportional to the difference of its local time of contact between $X_1$ and $X_2$. For a pair of independent Brownian motions $(B_1, B_2)$ both adapted to a continuous filtration $\mc{F}_t$, while $(X_1, Y, X_2, V)$ solves the following system with initial conditions $(x_1, y,x_2, v),$
% \begin{align*}
% \begin{split}
% dX_1 &= dB_1 - dL_1, \ dX_2 = dB_2 + dL_2\\
% \frac{dY(t)}{dt} &= v + K(L_2(t) - L_1(t)) =: V(t),\\
% L_i(0) &= 0, L_i \text{ is nondecreasing continuous, and flat off of $\{s : X_i(s) = Y(s)\}$},\\
% X_1(0) &= x_1 < y = Y(0) < x_2 = X_2(0),\\
% X_1(t) &\geq Y(t) \geq X_2(t) \text{ for all $t \in [0, T]$, almost surely.}
% \end{split}
% \end{align*}
% See example sample paths in Figure \ref{ch3:fig1}. Strong existence of this system was studied in \cite{white2007}, unfortunately the proof contains a nontrivial gap that we complete. Moreover, we use the discrete approximation scheme of processes in class $\mc{C}$, described above, to find an invariant measure of 
% $(X_1, Y, X_2, V)$.

% \begin{figure}
% \includegraphics[scale = .4]{untitled.eps}
% \includegraphics[scale = .4]{RBMvsRBMID1.eps}\label{ch3:fig1}
% \caption{Sample path simulations of reflected Brownian motion together with BMID. The drift of the latter becomes evident as time increases.}
% \end{figure}

\section{An equivalent formulation of BMID}\label{section:BMID}
\noindent
In this section we describe the process $(X, Y, V)$, where $X$ is Brownian motion reflecting from the inert particle $Y$ and where $Y$ has velocity $V$.
We begin with a probability space $(\Omega, \prob, (\mathcal{F}_t)_{t \ge 0})$, with the filtration $(\mathcal{F
}_t)_{t \ge 0}$ satisfying the usual conditions, supporting a Brownian motion $B$.
\begin{theorem}[Existence and Uniqueness, Knight \cite{Knight2001}, White \cite{white2007}]\label{th:BMID}
Choose $K \geq 0$ and $v \in \R$.
There exists a unique strong solution of continuous $\mc{F}_t$-adapted processes $(X, Y, V)$ satisfying:
\begin{align}
\begin{split}
& X(t) =  B(t) +  L(t), \text{ for all $t \geq 0$, almost surely,}\\
&X(t) \geq Y(t), \text{ for all $t \geq 0,$ almost surely,} \\
&d Y = V(t)dt := (v-KL(t))dt, \text{ for all $t \geq 0$, almost surely,}\\
&L \text{ is nondecreasing, and is flat away from the set } \{s : X(s) = Y(s)\}.
% &L(t) = \lim_{\e \to 0}\frac{1}{2\e}\int_0^t1_{[0, \e]}(X(s) - Y(s))ds \text{ for all $t,$ almost surely.}
\label{eq:sysLaw}
\end{split}
\end{align}
\end{theorem}
\noindent
% By flat we mean
% \[
% \int_\R1_{\{z : X(z) > Y(z)\}}(s)dL(s) = 0.
% \]
\begin{remark}\label{remark:BMID}
Flatness of $L$ off $\{s : X(s) = Y(s)\}$ means
\(
\int_0^\infty \mathbf{1}(X(s) \neq Y(s))\, \mathrm{d}L(s) = 0.
\)
One can use the Ito-Tanaka formula to show that $L$ is the local time of $X - Y$ at zero; see \cite[Th. 2.7]{white2007}. That is,
\[
L(t) = \lim_{\e \to 0}\frac{1}{2\e} \int_0^t\mathbf{1}(|X(s) - Y(s)| < \e)\, \mathrm{d}s,
\]
where the right hand side is the local time of $X - Y$ at zero.
\end{remark}

BMID together with its velocity is equivalent, in a certain sense, to the process $(X, Y, V)$ whose existence is given in Theorem \ref{th:BMID}.
We will refer to the following result by Skorohod.
\begin{lemma}[Skorohod, see \cite{KaratzasShreve}]\label{SkorohodLemma} Let $f \in C([0, T], \R)$
with $f(0) \geq 0.$ There is a unique, continuous, nondecreasing function
$m_f(t)$ such that
\begin{align*}
&x_f(t) = f(t) + m_f(t) \geq 0,\\
&m_f(0) = 0, \, m_f(t) \text{ is flat off } \{s : x_f(s) = 0\},
\end{align*}
that is given by 
\[
m_f(t) = \sup_{0 < s < t} [-f(s)] \lor 0.
\]
\label{classicSLemma}
\end{lemma}
\begin{remark}\label{remark:Flattness_Levy'sTheorem}
The classical L\'{e}vy's theorem states that for a Brownian motion $B$, $x_B$ is distributed as $|B|.$ See \cite[Section 3.6C]{KaratzasShreve}.
\end{remark}
% \begin{remark}
% L\'{e}vy's theorem says when $f$ is replaced by a Brownian motion, $B$. The corresponding process $x_B$ is distributed as $|B|.$
% \end{remark}
\noindent
To see the equivalence between BMID and the process $(X, Y, V)$ from Theorem \ref{th:BMID},
consider the gap process $G(t) = X(t) - Y(t)$. Obviously $G \geq 0$, almost surely, and from \eqref{eq:sysLaw} it follows that (when $v = 0$)
\begin{align}
dG = dB + dL + KLdt, \label{ch3:equivLawSDE}
\end{align}
where $L$ is continuous, nondecreasing, and flat off $U^{-1}(0).$ From the comment in Remark \ref{remark:BMID} on local time, $G$ is a reflected diffusion whose drift is proportional to its local time at zero. Consequently, the gap process $G$ is BMID as it satisfies \eqref{eq:SDE_BMID}.

Assume we have a pair of processes $(U, V)$ adapted to a continuous filtration $\mc{F}_t$ that supports a
Brownian motion $B$, and that for fixed $K \geq 0, v \in \R$
\begin{align}
\begin{split}
U(t) &= B(t) + \int_0^tV(x)\, \mathrm{d}x, \\
V(t) &= -v + KM^U(t), \\ 
M^U(t) &= \sup_{0< s < t}[-U(s)] \lor 0.
\label{ch3:eq:EquivSys}
\end{split}
\end{align}
In the system \eqref{ch3:eq:EquivSys}, it is clear from the definition of $M^U$ that 
$U + M^U \geq 0$ and it follows from the Skorohod Lemma \ref{SkorohodLemma} that $M^U$ is flat off the set $\{ s : U(s) + M^U(s) = 0\}.$
Therefore 
\[
B(t) + M^U \geq -\int_0^tV(s)\, \mathrm{d}s,
\] and $M^U$ is flat off of
$\{s : B(t) + M^U = -\int_0^tV(s)\, \mathrm{d}s \}.$ Consequently,
\[
\ds \left(B(t) + M^U(t), -\int_0^tV(s)\, \mathrm{d}s, -V(t)\right)
\]
satisfies the original equation \eqref{eq:sysLaw} with respect to the filtration $\mc{F}_t$. Similarly, one can use the uniqueness statement in Skorohod's Lemma to
go from a solution of \eqref{ch3:eq:EquivSys} to a solution of \eqref{eq:sysLaw}.
This demonstrates the equivalence of the two systems
\eqref{eq:sysLaw} and \eqref{ch3:eq:EquivSys} in the sense that if one solution exists for a given probability
space $(\Omega, \prob, \mc{F}_t)$, where $\mc{F}_t$ supports a given Brownian motion, then the other 
solution can be given by a path-by-path transformation.

% \begin{lemma}[Skorohod, see \cite{KaratzasShreve}] Let $f \in C([0, T], \R)$
% with $f(0) \geq 0.$ There is a unique, continuous, nondecreasing function
% $m(t)$ such that
% \begin{align*}
% &x(t) = f(t) + m(t) \geq 0,\\
% &m(0) = 0, \, m(t) \text{ is flat off of } \{s : x(s) = 0\},
% \end{align*}
% and is given by 
% \[
% m(t) = \sup_{0 < s < t} [-f(s)] \lor 0.
% \]
% \end{lemma}
% \begin{remark}
% As stated in the introduction, flatness off of $\{z : x(z) = 0\}$ for $m$ means 
% $\ds \int_0^t1_{\{z \, :\, x(z) > 0\}}(s)dm(s) = 0.$ The classical 
% L\'{e}vy's theorem says when $f$ is replaced by a Brownian motion, the corresponding process $x$ is distributed
% as $|B|.$
% \end{remark}
Existence of a solution to \eqref{eq:sysLaw} was first shown by Knight in \cite{Knight2001}. A strong solution to a more general process was attained via the employment
of a Skorohod map by David White \cite{white2007} in a more general version of Theorem \ref{theorem:Whites_smap} given below.
\begin{theorem}[White, \cite{white2007}]\label{theorem:Whites_smap}
For every $f \in C([0, T], \R), K \geq 0, v \in \R$ there is a unique pair of continuous 
functions $(I, V)$ such that
\begin{align}
\begin{split}
x(t) := f(t) + I(t),\\
V(t) = v + Km(t),\\
I(t) = \int_0^tV(s)\, \mathrm{d}s,\\
m(t) = \sup_{0 < s < t}[-x(s) \lor 0].
\end{split}
\end{align}
\end{theorem}
\begin{remark}
In Remark \ref{remark:Flattness_Levy'sTheorem} it is mentioned that replacing the function $f$ with a Brownian motion in the formulation
of Skorohod's Lemma \ref{SkorohodLemma} gives rise to a representation of reflected Brownian motion.
Similarly, when replacing $f$ in Theorem \ref{theorem:Whites_smap} pathwise
by Brownian motion, the corresponding process $(x, -I, V)$ is a solution to \eqref{ch3:eq:EquivSys}. Note that Skorohod's Lemma \ref{SkorohodLemma} implies that $m(t)$ in Theorem \ref{theorem:Whites_smap} is the unique monotonically increasing, continuous, function which is flat off of the level set $\{s : x(s) + m(s) = 0\}$ such that $x + m$ is nonnegative.
\end{remark}

\begin{remark}\label{remark:BMID_equiv_gap}
One can also see from the above arguments that $(U, V)$ of \eqref{ch3:eq:EquivSys} solves 
\[
dU(t) = dB(t) + dL(t) + V(t)dt, \ dV(t) = KL(t)dt, \ V(0) = -v,
\]
where $L$ is the local time of $U$ at zero.
\end{remark}

\section{Markov Processes with Memory}\label{ch3:section:MarkovProcessesMemory}
\noindent
The title of this section seems contradictory because Markov processes lose their memory when conditioning on their current location. The processes considered are pairs of processes, one process taking values in ``space," and the other process storing the history of the space-valued process. The transition rate of the space-valued process depends on this stored history.
We let $\mc{C}$ denote the class of such processes which we introduce more formally in this section. We will later construct a sequence of processes in $\mc{C}$ that will approximate BMID. First, we review well known facts of Poisson processes and point process with stochastic intensity.
For reference, see Br\'emaud's description of a doubly-stochastic point process in \cite[Chapter 2]{Bremaud}.

\subsection{Non-homogeneous Poisson processes}\label{ch3:subsection:point_processes}
% The Poisson distribution of mean $\ld$ is a random variables $Z$ with probability mass function 
% \[
% \prob(Z = n) = e^{-\ld}\frac{\ld^n}{n!}, \text{ for } n \in \N.
% \] A homogeneous Poisson process $N(t), t \geq 0$ with state space $\N$ and intensity (or rate) $\ld \in \R_+$ is a process such that
% \begin{enumerate}[label = (\roman*)]
% \item $N(0) = 0$,
% \item $N$ has independent increments,
% \item The number of jumps in an interval $(a, b)$ is Poisson r.v.\ with mean $\ld(b - a).$
% \end{enumerate}
% Typically we denote $N(a, b]$ as the number of jumps in the interval $(a, b].$
A non-homogeneous Poisson process with a nonnegative locally integrable rate (or intensity) function $\ld(t)$ is a process $N$ such that
\begin{enumerate}[label = (\roman*)]
\item $N(0) = 0$, a.s.
\item $N$ has independent increments,
\item $N$ is RCLL, a.s.
\item $N(a, b] = N(b) - N(a) \dist $ Poisson($\int_a^b\ld(s)$).
\end{enumerate}
%The next two lemmas follow readily from the definitions
If we let $T = \inf\{t :  N(t) > 0\}$ be the first jump time of $N$, then
\[
\ds \prob(T > t) = \prob(N(t) = 0)  = e^{-\int_0^t\ld(s)\, \mathrm{d}s}.
\]
% The following lemmas are classical, but we include a proof (or sketch) of the statements.
% \begin{lemma}\label{ch3:sumOfPoisson}
% Let $N_1, N_2$ be two independent Poisson processes with rate functions $\ld_1, \ld_2$. Then 
% $N_1 + N_2$ is a Poisson process with rate $\ld_1 + \ld_2.$
% \end{lemma}
% \begin{proof}
% Recall that the sum of two independent Poisson random variables of rates $a, b$ is itself a Poisson random
% variable of rate $a + b.$ Using this fact, the items (i)-(iv) are easily verified for $N_1 + N_2$ with a rate function of $\ld_1 + \ld_2.$
% \end{proof}

% \begin{lemma}\label{ch3:minJump}
% Let $N_1(t), N_2(t), t \geq 0$ be two independent Poisson processes with rate functions $\ld_1, \ld_2$, respectively. Let $T_1, T_2$ be the first jump time of $N_1, N_2$ respectively. Then $T_1 \land T_2$ is the first jump time
% of a Poisson process with rate function $\ld_1 + \ld_2.$
% \end{lemma}
% \begin{proof}
% This follows from Lemma \ref{ch3:sumOfPoisson}, because $T_1 \land T_2$ is the first jump of $N_1 + N_2.$
% % Clearly $T_1$ and $T_2$ are independent. Let $\Lambda_i(0, t) = \int_a^b\ld_i(s)ds$ for $i = 1, 2.$ For any $t > 0$,
% % \begin{align*}
% % \prob(T_1 \land T_2 > t) &= \prob(T_1 > t, T_2 > t)\\
% % &= e^{-\Lambda_1(0, t)}e^{-\Lambda_2(0, t)}\\
% % &= e^{-(\Lambda_1(0, t) + \Lambda_2(0, t))}\\
% % &= e^{-\int_0^t(\ld_1(s) + \ld_2(s))ds},
% % \end{align*}
% % proving the lemma.
% \end{proof}

\begin{lemma}\label{ch3:jumpRateStochDom}
Let $\ld_1, \ld_2$ be two rate functions such that $\ld_1(t) \leq \ld_2(t)$ for all $t \geq 0$ and
let $T_1, T_2$ be the first jump time of their corresponding Poisson process. Then $T_1$ stochastically dominates
$T_2.$
\end{lemma}
% \begin{proof}[Sketch]
% One could prove this simply from the explicit expression of the tail distribution. Alternatively, define $\ld_0 = \ld_2 - \ld_1$ and let $T_0$ be the first jump of a Poisson process with rate function $\ld_0$ independent from $T_1.$ Then $T_0 \land T_1$ has the same distribution as $T_2$. Hence there is a coupling of $(T_1, T_2)$ such that $T_1 \geq T_2$ almost surely. Hence $T_1$ stochastically dominates $T_2.$
% \end{proof}

% process, all on the same probability space. Suppose
% \[
% \ld_t \text{ is $\mc{F}_0$-measurable for $t \geq 0,$}
% \]
% and 
% \[
% \int_0^t\ld_sds < \infty, \ a.s. 
% \]
% for all $t \geq 0.$ If for all $0 \leq s \leq t$ and $u \in \R$ 
% \[
% \ex(e^{iu(N_t - N_s)}|\mc{F}_s) = \exp\left((e^{iu} - 1)\int_s^t\ld_vdv\right),
% \]
% then $N_t$ is an $\mc{F}_t$-doubly stochastic Poisson process or a Poisson process with stochastic intensity $\ld.$
\noindent
In \cite[Chapter II]{Bremaud}, Br\'emaud discusses the notion of point processes adapted to a filtration $\mc{F}_t$ whose intensity $\ld(s)$ is not a deterministic function but rather a process adapted to $\mc{F}_t$ with certain conditions.

\begin{defn}\label{ch3:Def:DoublySPP}\cite[II]{Bremaud}
Let $N_t$ be a point process adapted to the filtration $\mc{F}_t$ and let $\ld_t$ be a nonnegative $\mc{F}_t$-progressive process such that $\int_0^t\ld_s\, \mathrm{d}s < \infty$ almost surely for each $t \in [0, T]$. If
\begin{align}
\ex\left(\int_0^\infty C_s\, \mathrm{d}N_s \right) = \ex\left(\int_0^\infty C_s\ld_s\, \mathrm{d}s\right),\label{ch3:intensityDef}
\end{align}
for all nonnegative $\mc{F}_t$-predictable processes $C_t$
then we say $N_t$ has stochastic intensity $\ld_t.$
\end{defn}
\begin{remark}
In the proofs of later results we will refer to point processes with a given intensity or jump/step size. By a 
point process of jump/step size $a > 0$ and (stochastic) intensity $\ld$ we mean a process $aN_t$ where $N_t$ is a 
point process with (stochastic) intensity $\ld.$ By the \emph{positive} (resp. \emph{negative}) jump 
process for a process we mean the process $aN_t$ (resp. $-aN_t$). For example, a process with jump size $2^{-n}$
with positive jump rate $\ld_1(t)$, and negative jump rate $\ld_2(t),$ is $2^{-n}(N_1 - N_2)$ where $N_i$
is a point process with (stochastic) rate $\ld_i.$
\end{remark}
\noindent
Some well known facts of Poisson processes have analogous results for Poisson processes with stochastic intensities, which we list below.
\begin{lemma}\label{ch3:SumDoublySPP}
Let $N_1, N_2$ be two independent point processes with stochastic intensities $\ld^1, \ld^2$ 
adapted to filtrations $\mc{F}^1, \mc{F}^2$ respectively. Then $N_1 + N_2$ is a point process with stochastic intensity $\ld^1 + \ld^2$, adapted to $\mc{F}_t := \sigma\left(\mc{F}^1_t, \mc{F}^2_t\right).$
\end{lemma}
\begin{proof}[Sketch]
The fact that $N_1, N_2$ are independent implies the two processes
do not have common jumps, so that $(N_1, N_2)$ is a multivariate point process. The result follows from \cite[T15, Chapter II.2]{Bremaud}.
\end{proof}
\begin{lemma}\label{ch3:SPPstochDom}
Let $N_1$ be a point process with stochastic intensity $\ld^1(t) \geq \ld$, almost surely, for some $\ld \in \R_+.$ Then 
we can enlarge the probability space to support a Poisson point process $N_2$ with constant intensity $\ld$ and a point process $N_3$ of stochastic intensity $\ld^3 = \ld^1 - \ld$ such that $N_3$ is independent of $N_2$ and
 $N_2 + N_3$ has stochastic intensity $\ld^1$.
\end{lemma}
\begin{remark}
It is clear that one can generate a Poisson point process $N_2$ with constant intensity $\ld$ which
is independent of $N_1$. Lemma \ref{ch3:SPPstochDom} could be generalized to include more general lower bounds than a constant, however, we don't require this and sketch the proof only in the case when $N_2$ has constant intensity.
\end{remark}
\begin{proof}[Sketch]
Enlarge the probability space to support two independent processes $N_2', N_3'$ where $N_2'$ is a Poisson point process of rate $\ld$ and $N_3'$ is a point process with stochastic rate $\ld^3 = \ld^1 - \ld.$ By Lemma \ref{ch3:SumDoublySPP}, $N_1 \dist N_2 + N_3$, and the processes are adapted to the filtration generated by $(N_1, N_2, N_3).$
\end{proof}
% \begin{proof}
% By the tower property of conditional expectation and independence of 
% \end{proof}
% \begin{cor}
% \end{cor}

\subsection{Class $\mc{C}$ of Markov Processes with Memory} \label{Ch3:section:ClassC}
As mentioned, Burdzy and White \cite{Burdzy:White} study continuous time Markov processes $(X, L)$ on $\mathscr{L}\times\R$
where $\mathscr{L} = \{1, 2, \dots, N\}$ is a finite set. For each $j \in \mathscr{L}$
we associate a vector $v_j \in \R$ and define $L_j(t) = Leb(0< s < t : X(s) = j)$ as the
time $X$ has spent at location $j$ until time $t.$ We also define 
\[
L(t) = \sum_{j \in \mathscr{L}}v_jL_j(t)
\] as the accumulated time $X$ spends at each location, weighting the time spent at location $j$ by the factor $v_j$.
The transition rates of $X$ will depend on $L(t).$ More precisely,
we are given RCLL functions
\(
a_{ij} : \R \to \R
\) 
\noindent
where $a_{ij}$ is the rate function for the Poisson process defining the transition of $X$ from $i$ to $j.$ Conditional on $X(t_0) = i, L(t_0) = l,$ the jump rate of $X$ transitioning from $i$ to $j$ is $a_{ij}(l + [t-t_0]v_i)$ with $t \geq t_0.$ 
To construct such a process, for each $i$ we create independent random variables $(T_{i, j})_{j \in \mathscr{L}}$ which represents the jump time from $i$ to $j.$ Since this jump has intensity $a_{ij}(l + [t-t_0]v_i)$ with $t \geq t_0,$ 
\begin{align}\label{eq:CDF_JumpTimes}
\prob(T_{i, j} > t + t_0 | X(t_0) = i, L(t_0) = l) = \exp\left(-\int_0^ta_{ij}(l + sv_i)\, \mathrm{d}s\right),
\end{align}
for all $t > 0.$ Pick $j'$ such that $T_{i, j'} = \min_{j \neq i}T_{i, j}$, and define the first transition of
$X$ after time $t_0$ to be location $j'$ and occur at time $T_{i, j'}.$ \\

These dynamics can be produced from a collection of
independent exponential random variables $(E_{i, j})_{i, j \in \N}$ of rate one. Set $T_0 = 0$ and recursively define
\begin{align}\label{eq:Exponential_Rep_JumpTimes}
T_{i + 1}^j &= \inf\left( t > T_i: \int_0^ta_{X(T_i)j}(L(T_i) + v_{X(T_i)(s - T_i)}) \mathrm{d}s > E_{i, j}\right)\\
T_{i + 1} &= \min_{j}T_{i + 1}^j.
\end{align}
We use the convention $\inf \emptyset = \infty.$ Define
\begin{align}
L(s) &= L(T_i) + v_{X_n(T_i)}(s - T_i), \text{ for $s \in [T_i, T_{i + 1}]$}\\
X(s) &= X(T_i), \text{ for $s \in [T_i, T_{i + 1})$}\\
X(T_{i + 1}) &= \text{argmin}_{j}T_{i + 1}^j.
\end{align}

The pair $(X, L)$ is a strong Markov process with generator
\[
Af(j, l) = v_j \cdot \nabla_lf(j, l) + \sum_{i \neq j}a_{ji}(l)[f(i, l) - f(j, l)], \ j = 1, \dots, N,\ l \in \R.
\]
Burdzy and White assume $(X, L)$ is irreducible in the sense that there is some $\{j_0\} \times U \subset \mathscr{L}\times \R$ such that 
\[
\prob((X(t), L(t)) \in \{j_0\} \times U | X(0) = i, L(0) = l) > 0, \ \text{ for all $(i, l) \in \mathscr{L}\times \R$.}
\]
It should be noted that although they consider $\mathscr{L}$ to be a finite set, their main results hold
 assuming that $\sup_{ij}a_{ij}(l)$ is bounded on compact sets of $l$ and $\sup_{i}|v_i| < \infty$.
We denote $\mc{C}$ as the class of such processes with these conditions, allowing $\mathscr{L} = \N$. 

\section{Discrete Approximation}
\subsection{Definition of Processes}\label{subsection:def_discrete}
The reflected diffusion \eqref{eq:SDE_BMID} describing BMID is a process whose drift depends on the local time of the diffusion at zero. Intuitively, to approximate this diffusion with a Markov process on the lattice $2^{-n}\N$ one would  want the ``velocity'' to depend on the accumulated time spent at zero. This is modeled as
a jump process whose intensity function is stochastic and depends linearly on the accumulated time the process spends at zero. These jump processes need to converge, as $n \to \infty$, to a process whose drift is the appropriate local time. 

In Section \ref{section:BMID} we introduced an equivalent formulation for BMID given by $(U, V)$ in \eqref{ch3:eq:EquivSys}. In this subsection we will describe two equivalent discrete
processes that mirror the equivalence of the continuous processes described earlier; see Proposition \ref{remark:Equiv_jump_processes}. We do this because in order to prove the convergence result described in the introduction we actually
prove the convergence result for the equivalent formulation.

Jump processes whose intensity depends linearly on the accumulated time at zero are described by the class $\mc{C}$ in subsection \ref{Ch3:section:ClassC}.
Consider a process $(X_n, V_n^X)$ on the state space $2^{-n}\N \times \R$ where $v_j = 0$ for all $j \neq 0, v_0 = K2^n$ as given in the notation in that subsection.
(We may hide the dependence on $n$ for convenience.) 
% In other words, the rate will only depend on the time $X_n$ accumulates at zero.
For an initial ``velocity'' $v \in \R$, we define
\[
V_n^X(t) = -v + K2^nL_n(t) = -v + K2^n\cdot Leb(0 < s < t : X_n(s) = 0).
\] The rate functions $a_{ij}: \R \to \R$ are
\begin{align}
\begin{split}\label{ch3:rateFunctions}
a_{i(i + \sign(l)2^{-n})}(l) &= 2^{2n} + 2^n|l| = 2^{2n} + 2^n|V_n^X(t)|, \\
a_{i(i - \sign(l)2^{-n})}(l) &= 2^{2n}, \text{} 
\end{split}
\end{align}
where $l = V_n^X(t),$
except when $i = 0$ where we do not allow a downward transition.
By Lemma \ref{ch3:SumDoublySPP} the jump process $X_n$ can be decomposed into a sum of independent processes, $S_n$ and $Z_n$, whose rate functions sum to that of $X_n.$ The following definition will be used throughout the paper.
\begin{defn}\label{ch3:def:runningMin}
For a process $Q(t)$ we define $M^Q(t)$ as the signed running minimum below 
zero of $Q$. That is, 
\[
M^Q(t) = \sup_{0 < s < t}[-Q(s)] \lor 0.
\]
\end{defn}
\begin{defn}\label{def:SZV}
Consider the processes $(S_n, Z_n, V_n)$ on $\left(2^{-n}\Z\right)^2 \times \R$ where 
\begin{enumerate}[label = (\roman*)]
\item $S_n$ is a continuous time simple random walk on $2^{-n}\Z$ with positive (and negative) jumps of size $2^{-n}$ and rate $2^{2n}.$
\item $Z_n$ is a point process with jump size $2^{-n}$ and with positive (resp. negative) jumps having stochastic and adapted rate $2^n|V_n|$ when $V_n > 0$ (resp. $V_n < 0$).
\item We have
\begin{align*}
&V_n(t) = -v + K2^{n}\cdot Leb(0 < s < t : U_n(s) = -M_n(s)),\\
&U_n = S_n + Z_n,\\
&M_n(t) := M^{U_n}(t).
\end{align*}
\end{enumerate} 
That is, $S_n$ and $Z_n$ are point processes with adapted intensity functions as discussed in \cite[Chapter 2]{Bremaud}.
\end{defn}

Note the similarity to the equivalent formulation of BMID given by $(U, V)$ in \eqref{ch3:eq:EquivSys} to $(U_n, V_n)$ given above. Existence of $(S_n, Z_n, V_n)$ follows from the fact that it is of class $\mc{C}$, or, equivalently, one can construct the processes via the dynamics given in \eqref{eq:Exponential_Rep_JumpTimes} by using the intensity functions \eqref{ch3:rateFunctions}.
 
\begin{prop}\label{remark:Equiv_jump_processes}
The processes  $(U_n + M_n, V_n)$ and $(X_n, V_n^X)$ have the same law.
\end{prop}
\begin{proof}  With these definitions $(U_n + M_n, V_n)$ has the same law as $(X_n, V_n^X)$ because it is of class $\mc{C}$ and satisfies \eqref{ch3:rateFunctions}. To see this, note that $V_n$ is adapted to the right continuous filtration $\mc{F}_t$ generated by the pair $(S_n, Z_n).$ Also note that $U_n + M_n = S_n + Z_n + M_n$ is a nonnegative process on $2^{-n}\N$. By Lemma \ref{ch3:SumDoublySPP}, $S_n + Z_n$ has a jump rate function of $2^{2n} + 2^nV_n(t)$ where 
\begin{align*}
V_n(t) &= -v 
+ K2^{n}\cdot Leb\{0 < s < t : U_n(s) = -M_n(s)\}\\
&= -v +  K2^{n}\cdot Leb\{0 < s < t : U_n(s) + M_n(s) = 0\}.
\end{align*}
Consequently, $V_n^X = V_n$ if we define $X_n := U_n + M_n.$ Therefore
$(U_n + M_n, V_n)$ is one realization of the process $(X_n, V_n^X)$ given by \eqref{ch3:rateFunctions}.
\end{proof}

\noindent
We will work with $(S_n, Z_n, V_n)$ as an equivalent formulation of $(X_n, V_n^X)$ defined by \eqref{ch3:rateFunctions}.
\subsection{Theorem Statement}
% \begin{alignat*}{2}
% &U_n = S_n + Z_n, & U(t) = B(t) + \int_0^tV(x)dx,\\
% &M_n(t) := M^{U_n}(t), & M^U(t) = \sup_{0< s < t}[-U(s)] \lor 0,\\
% &V_n(t) = -v + K2^{n}\cdot Leb(0 < s < t : U_n(s) = -M_n(s)), & V(t) = -v + KM^U(t)\\
% \end{alignat*}
% \begin{align*}
% &V_n(t) = -v + K2^{n}\cdot Leb(0 < s < t : U_n(s) = -M_n(s)),   \qquad U(t) &= B(t) + \int_0^tV(x)dx, \\
% V(t) &= -v + KM^U(t), \\ 
% M^U(t) &= \sup_{0< s < t}[-U(s)] \lor 0.
% \label{ch3:eq:EquivSys}		\\
% &U_n = S_n + Z_n,    &				\\
% &M_n(t) := M^{U_n}(t).
% \end{align*} 
The main result of this article is that $(S_n, V_n, Z_n, U_n)$ converges in an appropriate sense to $(B, V, \int_0^\cdot V, U).$
% For motivation, see the similarities between the equivalent formulation of BMID given in \eqref{ch3:eq:EquivSys} to the definition of the processes $(S_n, Z_n, V_n, U_n)$ in Definition \ref{def:SZV} of subsection \ref{subsection:def_discrete}.
\begin{theorem}\label{ch3:thmDiscreteConvergence}
For $K \geq 0$ and $v \in \R$, let $(S_n, Z_n, V_n, U_n)$ be given as in Definition \ref{def:SZV} in subsection \ref{subsection:def_discrete}.
Then 
\[
(S_n, V_n, Z_n, U_n) \overset{d}{\lra} (B, V, \int_0^\cdot V, U),
\]
in the Skorohod topology on $D([0, T], \R^4)$, where $(B, V, \int_0^\cdot V, U)$ is a quadruple of continuous processes adapted to the Brownian filtration of the first coordinate $B$ with the following holding for all $t \in [0, T]$, almost surely:
\begin{align*}
\begin{split}
U(t) &= B(t) + \int_0^tV(x)\, \mathrm{d}x, \\
V(t) &= -v + KM^U(t), \\ 
\end{split}
\end{align*}
and where $M^U$ is the running minimum given in Definition \ref{ch3:def:runningMin}.
\end{theorem}
\noindent
Theorem \ref{ch3:thmDiscreteConvergence} has the following corollary.
\begin{cor}\label{th:cor_discrete_conv}
Let $(X_n, V_n^X)$ be the process defined by \eqref{ch3:rateFunctions} in subsection \ref{subsection:def_discrete}. Then
\[
(X_n, V_n^X) \overset{d}{\lra} (X, V),
\]
in the Skorohod topology on $D([0, T], \R)$.
Here $(X, V)$ is BMID together with its velocity. That is,
\begin{align}\label{eq:SDE_BMID_Vel}
dX(t) = dB(t) + dL(t) + V(t)dt, \ dV(t) = KL(t)dt, V(0) = -v,
\end{align} 
where $L$ is the local time of $X$ at zero.
\end{cor}\label{th:discrete_conv}
\begin{proof}[Proof of Corollary]
By Proposition \ref{remark:Equiv_jump_processes} we set $X_n := U_n + M_n$ and $V_n^X := V_n.$ Now
Theorem \ref{ch3:thmDiscreteConvergence} implies
\[
(X_n, V_n^X) \overset{d}{\lra} (U, -v + KM^U),
\]
and where $(U, -v + KM^U)$ solves \eqref{eq:SDE_BMID_Vel} as mentioned in Remark \ref{remark:BMID_equiv_gap}.
\end{proof}
\begin{remark}\label{ch3:remark:convInUnif}
Let $D([0, T], \R)$ denote the space of RCLL paths equipped with the Skorohod metric $d$ \cite[Chapter 3 Section 5]{EthierKurtz}. If a process $W_n$ with paths in $D([0, T], \R)$ converges weakly to $W$, then according to the Skorohod representation, \cite[Theorem 3.1.8]{EthierKurtz}, we can place $W_n, W$ on the same probability space such that
\[
d(W_n, W) \lra 0,
\]
almost surely. If the limiting process $W$ is continuous almost surely, then
\[
\|W_n - W\|_{[0, T]} := \sup_{0 < s < T}|W_n(s) - W(s)| \lra 0,
\]
almost surely on this probability space, and, in fact, uniform convergence and convergence in the Skorohod metric become equivalent. See Ethier and Kurtz, \cite[Chapter 3 Section 5]{EthierKurtz} and \cite[Chapter 3 Section 10]{EthierKurtz}, and Billingsley \cite[Chapter 3]{Billingsley}.
\end{remark}

\subsection{Proof of Theorem \ref{ch3:thmDiscreteConvergence}}
% \begin{lemma} 
% Assume we proved that for every $T > 0$, we have the following weak convergence in $D([0, T], \BR^3)$:
% \begin{equation}
% \label{eq:invariance-principle-one}
% (U_n, Z_n, S_n) \to (U, \int_0^tV(s)\,\md s, W).
% \end{equation}
% Then the main result: Theoremfollows.
% \end{lemma}

% \begin{proof} 
% The operator $x \mapsto \max_{0 \le s \le t}(-x(s))\lor 0$ is continuous in the space $D[0, T]$. Therefore, it follows from~\eqref{eq:invariance-principle-one} that $M_n \to M$ in $D[0, T]$. Plugging this together with~\eqref{eq:invariance-principle-one} into~\eqref{eq:new-X-Y}, ~\eqref{eq:U}, ~\eqref{eq:V-1}, we get:
% $$
% X_n \to W + M + x_0 = X,\ \ Y_n \to y_0 - \int_0^tV(s)\, \md s = Y.
% $$
% This completes the proof. 
% \end{proof}

In this section we prove Theorem \ref{ch3:thmDiscreteConvergence} assuming the two lemmas below, one for tightness and the other for classifying the subsequential limits.

\begin{lemma}[Tightness]
The collection of processes $\{(S_n, Z_n, V_n) : n \in \N\}$ is tight in $D([0, T], \R^2)$
 $\times C[0,T] \subset D([0, T], \R^3).$ Because $U_n = S_n + Z_n$, it follows that $\{(S_n,$ $V_n,$ $Z_n, U_n) : n \in \N\}$ is also tight in $D([0, T], \R) \times C[0, T] \times D([0, T], \R^2)$.
 Furthermore, all limiting processes are continuous.
 \label{ch3:lemma:tightness}
\end{lemma}
\noindent
We prove Lemma \ref{ch3:lemma:tightness}
in Section \ref{sec:tightness}. Assuming Lemma \ref{ch3:lemma:tightness} holds, it remains to show there is a unique limit.

\begin{lemma}[Classification of Limits]\label{lemma:sub}
Consider a subsequence $n_k$ with processes $(S_{n_k},$ $Z_{n_k},$ $V_{n_k},$ $U_{n_k})$
converging to $(S, Z, V, U)$ in $D([0, T], \R^4)$ with the Skorohod topology.
Then $(S, Z, V, U)$ is continuous and satisfies the equivalent formulation for
BMID given in \eqref{ch3:eq:EquivSys}. That is,
\begin{enumerate}[label = (\roman*)]
\item $U(t) = S(t) + Z(t),$
\item $S(t)$ is a Brownian motion,
\item $V(t) = KL(t) - v,$ where $L(t) = \max_{0< s < t}[-U(s)] \lor 0,$
\item $\ds Z(t) = \int\limits_0^tV(s)\, \mathrm{d}s.$
\end{enumerate} 
\end{lemma}
\noindent
Lemma \ref{lemma:sub} is proved in Section \ref{sec:sub}.

\begin{proof}[Proof of Theorem \ref{ch3:thmDiscreteConvergence}]
Since the formulation of BMID described by $(B, V, U)$ in the statement of Theorem \ref{ch3:thmDiscreteConvergence} is unique in law, Lemmas \ref{ch3:lemma:tightness} and \ref{lemma:sub} characterizes the subsequential limits of $(S_n, Z_n, V_n)$. See \cite{white2007} where existence (and uniqueness) of such a system is proved. Consequently, we have convergence of the entire sequence to this equivalent formulation of BMID. 
\end{proof}
\noindent

% \begin{defn}
% For a process $Q(t)$ we define $M^Q(t)$ as the signed running minimum below 
% zero of $Q$. That is, 
% \[
% M^Q(t) = \max_{0 < s < t}[-Q(s)] \lor 0.
% \]
% \end{defn}

\subsection{Lemma \ref{ch3:lemma:tightness}: Tightness of $(S_n, Z_n, V_n)$}\label{sec:tightness}
% Denote our probability space as $(\Oa, \prob, \mc{F}_t)$
Recall that our process $(S_n, Z_n, V_n)$ is in $D([0, T], \R^3)$, the space of RCLL paths
with the Skorohod topology defined by the product metric $d \times d \times d$ where $d$ is the Skorohod metric, see Billingsley \cite{Billingsley}. The following definition is taken from Jacod and Shiryaev \cite{Shiryaev}.

\begin{remark}\label{remark:Skorohod_sum}
In general, it is not true that if $\alpha_n \to \alpha$, $\beta_n \to \beta$ in the Skorohod topology, then $\alpha_n + \beta_n \to \alpha + \beta.$ However, this does hold if either $\alpha$ or $\beta$ is continuous. See Jacod and Shiryaev \cite[Proposition VI.1.23]{Shiryaev}. Similarly, by Remark \ref{ch3:remark:convInUnif} one can assume a sequence $(S_{n_k}, Z_{n_k}, V_{n_k})$ that converges in distribution on $D([0, T], \R^3)$ in fact converges almost surely to a continuous process, say $(S', Z', V')$ in the uniform metric if the limit is continuous; at least on some probability space. This immediately implies $U_{n_k} = S_{n_k} + Z_{n_k}$ converges almost surely to $S' + Z'.$
\end{remark}

\begin{defn}{\cite[Definition VI.3.25]{Shiryaev}}\label{ch3:}
A sequence of processes $\{X_i : i \in \N\}$ in $D([0, T], \R)$ is said to be $C$-tight
if $\{X_i : i \in \N\}$ is tight on $D([0, T], \R)$ and all limiting processes are continuous.
\end{defn}
\noindent
See Remark \ref{ch3:remark:convInUnif}, which mentions the Skorohod metric when the limit processes are continuous. The proof that $(S_n, Z_n, V_n)$ is tight is broken into
multiple lemmas. Recall that $L_n(t) := 2^{n}Leb(0 < s < t : U_n(s) = -M_n(s))$ where $U_n := S_n + Z_n$, $M_n := M^{U_n},$ and $V_n := -v + KL_n$.

\begin{defn}
For $f \in D([0, T], \R)$, let
\[
\omega(f, \d) := \sup_{\substack{0 < s < t < T \\ |t - s| < \d}}|f(t) - f(s)|
\]
be the modulus of continuity of $f.$
\end{defn}
\noindent
Recall that $\|f\|_{[a, b]} = \sup_{a< x < b}|f(x)|.$

\begin{lemma}{\cite[Proposition VI.3.26]{Shiryaev}}\label{ch3:LemmaCTight}
A sequence of processes $X_n$ in $D([0, T], \R)$ is $C$-tight if and only if for every $\e > 0$,
\begin{enumerate}[label = (\roman*)]
\item $\lim_{C \to \infty}\limsup_{n \to \infty}\prob(\|X_n\|_{[0, T]} \geq C) = 0$,\\
\item $\lim_{\d \to 0}\limsup_{n \to \infty}\prob(\omega(X_n, \d) > \e) = 0$.
\end{enumerate}
\end{lemma}
\begin{remark}\label{ch3:remark:Ctight}
Notice that (i) follows from (ii) and $\lim_{C \to \infty}\limsup_{n \to \infty}\prob(|X_n(0)| > C) = 0.$ 
To see this, take $\d = 1$ and $\e = C/2$ so that by definition of the modulus of continuity
\[
\|X_n\|_{[0, T]} \leq |X_n(0)| + \sum_{i = 1}^T\omega(X_n, 1).
\]
Consequently, the triangle inequality gives
\begin{align*}
\prob(\|X_n\|_{[0, T]} > C) &\leq \prob(|X_n(0)| + \lceil T\rceil\omega(X_n, 1) > C)\\
&\leq \prob(|X_n(0)| > C/2) + \prob(\omega(X_n, 1) > C/(2\lceil T \rceil)),
\end{align*}
where $\lceil r\rceil$ is the smallest integer larger than $r,$
from which it is clear that (ii) and $\lim_{C \to \infty}\limsup_{n \to \infty}\prob(|X_n(0)| > C) = 0$ imply (i).
\end{remark}

\begin{lemma} Assume that the sequences of processes $(X_n), (X'_n), (X''_n)$ in $D([0, T], \R)$ satisfy
$$
X'_n(t) - X'_n(s) \le X_n(t) - X_n(s) \le X''_n(t) - X''_n(s),\ \ 0 \le s \le t,
$$
almost surely.
If both $(X_n')$ and $(X_n'')$ are $C$-tight, then $(X_n)$ is also $C$-tight.
\label{lemma:criterion}
\end{lemma}
\begin{proof}
 For any $C > 0$, the triangle inequality gives
\[
\prob(\|X_n\|_{[0, T]} > C) \leq \prob(\|X_n'\|_{[0, T]} > C/2) + \prob(\|X_n''\|_{[0, T]} > C/2),
\] which verifies condition (i) in the statement of Lemma \ref{ch3:LemmaCTight} by taking $\lim_{C \to \infty}$ $\limsup_{n \to \infty}$ of both sides.
Similarly, for every $\d, \e > 0,$
\[
\prob(\omega(X_n, \d) > \e) \leq \prob(\omega(X_n', \d) > \e/2) + \prob(\omega(X_n'', \d) > \e/2),
\]
and taking $\lim_{\d \to 0}\limsup_{n \to \infty}$ on both sides verifies condition (ii) in the statement of Lemma \ref{ch3:LemmaCTight}.
\end{proof}

The following two lemmas are classical and we omit proofs.

\begin{lemma}\label{ch3:lemma:minOfExpsIndofIndicator}
Let $X \dist \Exp(\ld), Y \dist \Exp(\mu)$ be independent. Then $X \land Y \dist \Exp(\ld + \mu)$ is
independent from $W = 1_{\{X \land Y = X\}}$. In other words, the minimum of two independent exponential random variables is independent from which exponential r.v.\ occurred first.
\end{lemma}

% \begin{proof}
% Because $X, Y$ are independent we can write their joint pdf as
% $f(x, y) = \ld\mu\exp(-\ld x - \mu y)$. With this, one can compute, for any $z > 0$,
% \begin{align*}
% \prob(X \land Y > z | W = 1) &= \prob(X \land Y > z | X \land Y = X)\\
% &=\int_z^\infty\int_x^\infty f(x, y)\, \mathrm{d}y\, \mathrm{d}x\cdot \frac{1}{\prob(X \land Y = X)}\\
% &= \frac{\ld}{\mu + \ld}e^{-(\ld + \mu)z}\cdot \frac{\mu + \ld}{\ld}\\
% &= e^{-(\ld + \mu)z}\\
% &= \prob(X \land Y > z).
% \end{align*}
% \end{proof}

\begin{lemma}\label{ch3:lemma:functionalLLN}
Fix $a > 0$, and for each $n \in \N$ let $N_n$ be a Poisson process with intensity $\alpha2^n$. Then $\{2^{-n}N_n(t) : t \in [0, T]\}$
converges in distribution to the line $g(t) = a t$, in the space $D([0, T], \R)$. In particular $\{2^{-n}N_n(t) : t \in [0, T], n \in \N\}$
is $C$-tight.
\end{lemma}
% \begin{proof}
% $N_n(t)$ is a Poisson process with rate $C2^n$, so 
% $N_n(t) - C2^nt = N_n(t) - 2^ng(t)$ is a martingale. Consequently $2^{-n}(N_n(t) - 2^ng(t)) = 2^{-n}N_n(t) - g(t)$
% is also a martingle. By Doob's martingale inequality and Cauchy-Schwarz, for any $\e > 0$
% \begin{align*}
% \prob(\|2^{-n}N_n(t) - g(t)\|_{[0, T]} > \e) &\leq \e^{-1}\ex|2^{-n}N_n(T) - g(T)|\\
% &\leq \e^{-1}\sqrt{\Var(2^{-n}N_n(T))} \\
% &= \e^{-1}\sqrt{2^{-2n}C2^n}\\
% &\lra 0.
% \end{align*}
% Thus $2^{-n}N_n(t)$ converges weakly in $D([0, T], \R)$ to $g$ because convergence in the uniform norm implies convergence in the Skorohod topology on $D([0, T], \R)$, \cite{Billingsley}.
% \end{proof}

\begin{lemma}\label{ch3:Lemma:LandZ'boundsOnExtendedSpace}
There is a filtered probability space $(\Omega, (\F_t)_{t \geq 0}, \prob)$ satisfying the usual conditions, supporting the $\F_t$-adapted process $(S_n, Z_n, V_n)$ given in Definition \ref{def:SZV}, also supporting the $\F_t$-adapted process $(U_n', Z_n', L_n')$, such that
\begin{align*}
U_n' &= S_n + Z_n',\\
M_n' &= M^{U_n'},\\
L_n'(t) &= 2^{n}Leb(0 < s < t : U_n' = -M_n').
\end{align*} Here $Z_n'$ is a Poisson point process of intensity $|v2^n|$ and jump size $\sign(v)2^{-n}$. Furthermore,
\begin{align}
Z_n'(t) - Z_n'(s) \leq Z_n(t) - Z_n(s), \label{ch3:LemmaZ'Ineq}
\end{align}
for all $0 \leq s \leq t \leq T$, almost surely, and
\begin{align}
0 \leq  L_n(t) - L_n(s) \leq L_n'(t) - L_n'(s), \label{ch3:LemmaL'Ineq}
\end{align}
for all $0 \leq s \leq t \leq T$, almost surely. The construction will yield independence between $Z_n'$ and $S_n.$
\end{lemma}
\begin{proof}
By definition $V_n \geq -v$ almost surely. Recall that $Z_n$ is a point process with stochastic intensity $|2^nV_n|$ and jump size $\sign(V_n)2^{-n}$, so by Lemma \ref{ch3:SPPstochDom} we assume the probability space included a process $Z_n'$ with downward stochastic jump intensity $|v2^{n}|$ and step size $2^{-n}$ such that 
\[
Z_n'(t) - Z_n'(s) \leq Z_n(t) - Z_n(s)
\]
for all $0 \leq s \leq t \leq T,$ almost surely. This inequality holds because the negative transitions of $Z$ will have a rate less than $|v|2^n$, which is the transition rate for negative jumps of $Z_n'$. (And by definition $Z_n'$ makes negative jumps only.) The jump times of $Z_n'$ are independent of $S_n$ hence the processes are independent. This demonstrates \eqref{ch3:LemmaZ'Ineq}. It follows that
 $$U_n := S_n + Z_n \geq S_n + Z_n' =: U_n',$$
 almost surely. Notice both processes $U_n + M_n, U_n' + M_n'$ transition as a continuous time 
 (nonnegative) random walk but with an additional ``drift'' process of $Z_n, Z_n'$ respectively.
 For instance, a transition of $U_n$ beginning from its running minimum corresponds to a transition from zero for
 the walk $U_n + M_n$. By \eqref{ch3:LemmaZ'Ineq}, the process $U_n+M_n^U$ dominates that of $U_n' + M_n^{U'}$. That is,
 \begin{equation}
 U_n + M_n \geq U_n' + M_n' \geq 0, \text{ almost surely.}
 \label{eq:b2}
 \end{equation}
 Hence, $U_n' + M_n'$ is zero whenever $U_n + M_n$ is zero. Consequently,
 $\{ s < z < t :  U_n'(z) = - M_n'(z)\} \subset \{z : s < z < t,  U_n(z) = - M_n(z)\}$
 for every $(s, t) \subset [0, T],$ almost surely. Therefore,
 \begin{equation}
 0 \leq L_n(t) - L_n(s) \leq L_n'(t) - L_n'(s),
 \label{eq:b3}
 \end{equation}
for every $(s, t) \subset [0, T]$, almost surely, demonstrating \eqref{ch3:LemmaL'Ineq}.
\end{proof}

\begin{lemma}\label{ch3:momentBoundMn}
For every $T>0,$ $\ex(M_n(T)) \leq \ex(M_n'(T))\leq 2\sqrt{2T + T|v|\sqrt{2T} + 2|v|T}$ where $v$ is the initial value of $V_n$ and $M_n'$ is defined in Lemma \ref{ch3:Lemma:LandZ'boundsOnExtendedSpace}.
\end{lemma}
\begin{proof}
According to Lemma \ref{ch3:Lemma:LandZ'boundsOnExtendedSpace} we assume our probability space supports $(S_n, Z_n, V_n)$ as well as the $Z_n'$ given in that lemma's statement. Consequently,
\[
U_n := S_n(t) + Z_n(t) \geq S_n(t) + Z_n'(t) =: U_n',
\]
for all $t \in [0, T]$, almost surely, implying
\begin{align}\label{eq:Bound_M_n}
M_n(T) := M^{U_n}(T) \leq M^{U_n'}(T) =: M_n'(T), 
\end{align}
almost surely. We can express the continuous time random walk $S_n$ as $2^{-n}(N_1(t) - N_2(t))$ where $N_i$ are independent 
Poisson processes of rate $2^{2n}$, and consequently $\ex(S_n(T)^2) = \Var(S_n(T)) = 2^{-2n}(\Var(N_1(T)) + \Var(N_2(T))) = 2T.$ By Cauchy-Schwarz this yields $\ex|S_n(T)| \leq \sqrt{2T}.$ By Doob's Martingale inequality, the fact that $Z_n'(T)$ is distributed as $-2^{-n}$ $\text{Poisson}(|v2^nT|)$, and independence of $Z_n'$ and $S_n$, we compute

\begin{align*}
	\ex(M_n(T)) &\overset{\eqref{eq:Bound_M_n}}{\leq} \ex(M_n'(T))\\
	&\leq \sqrt{\ex(M_n'(T)^2)}, \text{       Cauchy-Schwarz}\\
	&\leq\sqrt{4\ex(U_n'(T)^2)} = 2\sqrt{\ex(|S_n(T) + Z_n'(T)|^2)}, \text{by Doob's Maximal Inequality}\\
	&\leq 2\sqrt{\ex(S_n^2(T)) + \ex|S_n(T) Z_n'(T)| + \ex (Z_n'(T)^2)} \leq 2\sqrt{2T + T|v|\sqrt{2T} + 2|v|T}.
\end{align*}

\end{proof}

\begin{lemma}
In the notation of Lemma \ref{ch3:Lemma:LandZ'boundsOnExtendedSpace}, $L_n'$ converges in distribution to $M^{B^{-v}}$ in the space 
$C([0, T], \R)$ with the uniform norm. Here,
$B^{-v}(t) := B(t) - vt$ where $B$ is a Brownian motion. In particular, $\{L_n' : n \in \N\}$ is $C$-tight. Furthermore,
$\sup_{n} \ex L_n(T) \leq \sup_n \ex L_n'(T) < \infty$.
\label{lemma:L_nBound}
\end{lemma}
\begin{proof}
We begin by showing the weak convergence for which we use a similar technique in the proof of Lemma \ref{lemma:sub} (iii). We record the amount of time $U_n'$ spends at each level of the running minimum, and express $L_n'(t)$ as the sum of these times.
By Lemma \ref{ch3:lemma:functionalLLN}, $Z_n'(t) = -2^{-n}N_n(t)$ 
converges in distribution to $g(t) = -vt$ in the space $D([0, T], \R)$. By Donsker's theorem, $S_n$ converges 
in distribution to a Brownian motion $B$ and consequently $U_n' := S_n + Z_n'$ converges in distribution to $B + g =: U'.$ This implies $M_n' := M^{U_n'}$ converges to $M^{U'}$ in distribution because $f \mapsto M^f$ is continuous in the uniform norm and the limiting processes are continuous. See Remark \ref{remark:Skorohod_sum}.
% (This is classical).
 Note that $2^nM_n'(t) + 1$ is the number of levels the running minimum of $U_n'$ has reached by time $t.$
% $U_n' = M_n^{U'}$ at least for some time.
 Let $\tau^{(j)} = \inf\{t > 0: U_n'(t) = -j2^{-n}\}$, so that $\tau^{(j)}$ is the first time $M_n'$ 
 reaches $j2^{-n}.$ Define
$$
T_j := Leb(s \ge \tau^{(j)}\mid -U_n'(s) = M_n'(s) = j2^{-n}).
$$
Then,
\begin{align}
\begin{split}
2^{n}\sum_{j=0}^{2^nM_n'(s)}T_j&\leq L_n'(s)
 \leq 2^{n}\sum_{j=0}^{2^nM_n'(s)+1}T_j, \label{ch3:eqL'squeeze}
 \end{split}
\end{align}
for all $s \in [0, T]$, almost surely. When $M_n' = j2^{-n}$,
after the process $U_n'$ arrives at $-j2^{-n}$ for the $k$th time, it makes a positive jump upon the arrival of an $\Exp(2^{2n})$ random variable, call it $\mu_k^{(j)}$, while it
makes a negative jump upon the arrival of an $\Exp(2^{2n} + |v2^n|)$ random variable $\mu_k^{(j)'}$. Consider the pair $(\mu_k^{(j)} \land \mu_k^{(j)'}, I_j)$ where $I_j = 1_{\{\mu_k^{(j)} \land \mu_k^{(j)'} = \mu_k^{(j)}\}}$. By Lemma \ref{ch3:lemma:minOfExpsIndofIndicator}, $I_j$ is independent from the i.i.d.\ sequence $(\mu_k^{(j)} \land \mu_k^{(j)'} : k \geq 1)$. Then, $W_j := \inf\{k : I_k = 0\}$ is the number of times $U_n'$ visits $-j2^{-n}$ while the signed running minimum $M_n'$ is $j2^{-n}$. Because $I_j \dist \text{Bernoulli}(p)$ with 
\[
p = \frac{2^{2n}}{2^{2n+1} + v2^n},
\] and $W_j$ is Geometrice$(p)$ (since it is the first time this sequence of Bernoulli random variables is zero), and Lemma \ref{ch3:lemma:minOfExpsIndofIndicator} implies that $W_j$ is independent of $\mu_i^{(j)}, \mu_i^{(j)'}$. Thus,
\[
T_j = \sum_{i = 1}^{W_j}\mu_i^{(j)} \land \mu_i^{(j)'}
\]
is a Geometric sum of i.i.d.\ exponential random variables of rate $\la = 2^{2n+1} + v2^n$ that are independent of the number of the sums $W_j$. Such a sum is exponential
of rate $p\ld$.
That is, $T_j \dist \Exp(p\ld) = \Exp(2^{2n})$. Each $T_j$ is measurable with respect to $\sigma\{(U_n'(s), M_n'(s)): s \in (\tau^{(j)}, \tau^{(j + 1)}]\},$ and $(T_j : j \geq 1)$ is independent of $(U_n'(\tau^{(j)}), M_n'(\tau^{(j)}))$. (In other words, $T_j$ depends only on the excursions between the stopping times and not the initial position.) Thus $(T_j : j \geq 1)$ is a sequence of i.i.d.\ $\Exp(2^{2n})$ random variables. We show the left hand side of \eqref{ch3:eqL'squeeze} converges in probability to 
 $M'(s)$ in the uniform norm for $s \in [0, T]$. The proof for the right hand side is essentially identical.
% \begin{align*}
% \sup_{s \in [0, T]}\left|\sum_{i = 1}^{2^nM_n'(s)}2^{-n}  - M_n'(s)\right| \leq 2^{-n}.
% \end{align*}
Without loss of generality, we may assume $M_n'$ converges almost surely to $M'$ by the Skorohod representation theorem and the fact shown above that $M_n'$ converges to $M'$ in distribution. Therefore 
\[
\sup_{s \in [0, T]}\left|\sum_{j = 0}^{2^nM_n'(s)}2^{-n} - M'(s)\right| \lra 0, \text{ almost surely.}
\] 
To show the sequence $\ds 2^n\sum_{j = 0}^{2^nM_n'(s)}T_j, n \geq 1,$ converges in probability to 
$M'(s)$ uniformly for $s \in [0, T],$ it suffices to show 
\[
\sup_{s \in [0, T]}\left|2^n\sum_{j = 0}^{2^nM_n'(s)}T_j - \sum_{j = 0}^{2^nM_n'(s)}2^{-n}\right| \overset{\prob}{\lra} 0.
\]

Because $T_j \dist \text{Exp}(2^{2n})$, we know $z_j := 2^nT_j \dist \text{Exp}(2^n).$ Then,
\[
2^n\sum_{j = 0}^{2^nM_n'(s)}T_j - \sum_{j = 0}^{2^nM_n'(s)}2^{-n} = \sum_{j = 0}^{2^nM_n'(s)}(z_j - 2^{-n}),
\]
where $z_j - 2^{-n}$ are i.i.d.\ mean zero random variables with variance $2^{-2n}$. By Kolmogorov's maximal inequality, for each $C, \e > 0$
\begin{align*}
&\prob\left(\sup_{s \in [0, T]}\left|\sum_{j = 0}^{2^nM_n'(s)}(z_j - 2^{-n})\right| > \e\right)\\
&\leq \prob\left(\sup_{1 \leq k \leq 2^nC}\left|\sum_{j = 0}^k(z_j - 2^{-n})\right| > \e\right) + \prob(M_n'(T) > C)\\
% &\leq \prob\left(\sup_{1 \leq k \leq 2^nC}\sum_{i = 1}^k|z_j - 2^{-n}| > \e\right) + \prob(M_n'(T) > C)\\
&\leq \e^{-2}\Var\left( \sum_{j = 0}^{2^nC}(z_j - 2^{-n})\right) + \prob(M_n'(T) > C)\\
&\leq \e^{-2}2^nC2^{-2n} + \prob(M_n'(T) > C).
\end{align*}
Because $M_n'$ converges to $M'$ almost surely, this implies
\begin{align*}
&\limsup_{n \to \infty}\prob\left(\sup_{s \in [0, T]}\left|\sum_{j = 0}^{2^nM_n'(s)}(z_j - 2^{-n})\right| > \e\right)\\
&\leq \limsup_{n \to \infty}\prob(M_n'(T) > C)\\
&= \prob(M'(T) > C),
\end{align*}
where $C > 0$ is arbitrary. Since $M'(T)$ is finite a.s., $\prob(M'(T) > C)$ can be made arbitrarily small with a large choice for $C$. Hence 
\begin{align*}
\sup_{s \in [0, T]}\left|\sum_{j = 0}^{2^nM_n'(s)}(z_j - 2^{-n})\right| \overset{\prob}{\lra} 0,
\end{align*}
and the left hand side of \eqref{ch3:eqL'squeeze} converges in probability to $M'$. That is,
\begin{align}\label{eq:lemma_conv_prob_L}
\sup_{s \in [0, T]}\left|2^{n}\sum_{j=0}^{2^nM_n'(s)}T_j - M'(s)\right| \overset{\prob}{\lra} 0.
\end{align}
The convergence in probability for the right hand side is similar,
\begin{align}\label{eq:lemma_conv_prob_R}
\sup_{s \in [0, T]}\left|2^{n}\sum_{j=0}^{2^nM_n'(s) + 1}T_j - M'(s)\right| \overset{\prob}{\lra} 0.
\end{align}
Because \eqref{ch3:eqL'squeeze} is an almost sure bound,
\begin{align*}
&\prob\Big(\sup_{s \in [0, T]}|L_n'(s) - M'(s)| > \e\Big)\\ &\leq \prob\left(\sup_{s \in [0, T]}\left|2^{n}\sum_{j=0}^{2^nM_n'(s)}T_j - M'(s)\right| > \e/2\right) + \prob\left(\sup_{s \in [0, T]}\left|2^{n}\sum_{j=0}^{2^nM_n'(s) + 1}T_j - M'(s)\right| > \e/2\right),
\end{align*}
and taking $\limsup_{n \to \infty}$ on both sides, \eqref{eq:lemma_conv_prob_L} and \eqref{eq:lemma_conv_prob_R} imply
\[
\limsup_{n \to \infty}\prob\Big(\sup_{s \in [0, T]}|L_n'(s) - M'(s)| > \e\Big) = 0,
\] for every $\e > 0.$ This complete the proof that $L_n'$ converges in probability to $M'$ in the uniform norm. Since $M'$ is a continuous process, the sequence $\{L_n' : n \in \N\}$ is $C$-tight; see Definition \ref{ch3:}.

To demonstrate the uniform moment bound, note that $L_n(T) \leq L_n'(T)$, equation \eqref{ch3:eqL'squeeze}, and Wald's lemma imply
\[
\ex(L_n(T)) \leq \ex(L_n'(T)) = \ex(M_n'(T)) + 2^{-n}.
\]
Applying the uniform moment bound on $M_n'(T)$ given in Lemma \ref{ch3:momentBoundMn}, we see
\[
\sup_n\ex(L_n(T)) \leq \sup_n\ex(L_n'(T)) \leq \sup_n\ex(M_n'(T)) + 1 < \infty.
\]
\end{proof}

% Therefore, $M_n'(t) \Ra M^{B^{-v}}(t)$, and with high probability for large $n$, we have $M_n'(t) \le C$ for some constant $C$. But by the standard law of large numbers, 
% \[
% 2^{-n}\SL_{j=1}^{2^nC}\left(\xi_j - 1\right) \Ra 0.
% \]
% For the moment bound, notice that $\ex M_n'(t) \leq \ex M_n^S(t) + \ex |Z_n'(t)|$, where
% $M_n^S = M^{S_n}$ is the signed running minimum of the simple random walk $S_n.$ By definition $\ex |Z_n'(\cdot)| = 2^{-n}\ex N_n(t) = vt$, and by Doob's
% martingale inequality there is a constant $A > 0$ such that
% \[
% \ex M_n'(t) \leq \ex M_n^{S}(t) + \ex|Z_n'(t)| \leq A\sqrt{t} + vt < \infty,
% \]
% completing the moment bound.
% \begin{lemma} Assume that the sequences $(X_n), (X'_n), (X''_n)$ are in $D[0, T]$ $(resp. C[0, T])$, the last two are tight in the Skorohod topology (resp. uniform norm), and 
% $$
% X'_n(t) - X'_n(s) \le X_n(t) - X_n(s) \le X''_n(t) - X''_n(s),\ \ 0 \le s \le t.
% $$
% Then $(X_n)$ is also tight in $D[0, T]$  $(resp. C[0, T])$. If subsequential
% limits of $X'_n, X''_n$ are in $C[0, T],$ then so are subsequential limits of $X_n.$
% \label{lemma:criterion}
% \end{lemma}
\begin{cor}\label{ch3:corL_ntight}
The collection of processes $\{L_n : n \in \N\}$ is $C$-tight.
\end{cor}
\begin{proof}
This follows directly from \eqref{eq:b3}, the fact that $\{L_n' : n \in \N\}$ is $C$-tight by Lemma \ref{lemma:L_nBound} (and that the zero process is trivially $C$-tight), and Lemma \ref{ch3:LemmaCTight}.
\end{proof}
\begin{lemma}\label{ch3:Z_nIsTight}
The collection of processes $\{Z_n : n \in \N\}$ is $C$-tight.
\end{lemma}
\begin{proof}
We will use a localization argument by stopping the 
stochastic intensity of $Z_n$ when it becomes large. Recall that $|2^nV_n|$ is the stochastic intensity of $Z_n.$
For $C > v \geq 0$ set
\[
T_C^{n} = \inf\{ t > 0: V_n(t) > C\} = \inf\left\{t > 0: L_n(t) > \frac{v + C}{K}\right\}.
\]
%   Recall in the definition of $(S_n, Z_n ,V_n)$
% $$
% Z_n(t) = 2^{-n}\SL_{k:\, \tau_k \le t}1(\eta_k < \zeta_k)\sign V(\tau_k).
% $$
% If $\tau_k \le T_\e^{(n)}$, then $\eta_k$ dominates $\Exp(C2^n)$. 
Define  a process $\wt{Z}_n$ such that $\wt{Z}_n|_{[0, T_C^n]} = Z_n|_{[0, T_C^n]}$, almost surely, while after time $T_C^n$ let $\wt{Z}_n$ have positive jump intensity $C2^n$ and
jump size $2^{-n}$ (and make only positive jumps).
By Lemma \ref{ch3:SPPstochDom} we can stochastically dominate the number of transitions made by  
 $\wt{Z}_n$ in a given time interval by the number of transitions made by a point process $Z^1_n(t)$ 
 of intensity $C2^n$ and jump size $2^{-n}$ (in the same time interval). More precisely, we may assume there exists a process $Z^1_n(t) = 2^{-n}N(C2^nt)$ on our probability space, where $N$ is Poisson process of unit intensity, and
\begin{align}
0 \leq |Z_n(t) - Z_n(s)| \leq |Z_n^1(t) - Z_n^1(s)| \label{ch3:ZnTightnessEq1}
\end{align}
for every interval $(s, t) \subset [0, T_C^{(n)}],$ almost surely. 
% By \eqref{ch3:ZnTightnessEq1} and Lemma \ref{lemma:criterion} it suffices to show $Z_n^1$ satisfies (i) and (ii) of Lemma \ref{lemma:criterion}.
By monotonicity of $Z_n^1$ and Doob's maximal inequality, for fixed $\e, \delta, C > 0$ we have

\begin{align*}
&\prob\Big(\sup_{t \in [0, T -\delta]}\sup_{t < u,v < t + \delta}|Z_n(u) - Z_n(v)| > \e\Big)\\
&\leq \prob\Big(\sup_{t \in [0, T -\delta]}\sup_{t < u,v < t + \delta}|\wt{Z}_n(u) - \wt{Z}_n(v)| > \e, T_C^{(n)} > T\Big) + P(T_C^{(n)} < T)\\
&\leq \prob\Big(\sup_{t \in [0, T-\delta]}\sup_{t < u, v < t + \delta}|Z_n^1(u) - Z_n^1(v)| >\e\Big) + P(T_C^{(n)} < T)\\
&= \prob\Big(\sup_{t \in [0, T- \delta]}|Z_n^1(t+\delta) - Z_n^1(t)| > \e\Big) + P(V_n(T) >
 C)\\
% &\leq\prob\Big(\sup_{t \in [0, T- \delta]}|Z_n^1(t+\delta) - Z_n(t)| > \e\Big) + P(T_C^{(n)} > T)\\
% &= \e^{-1}\delta CT + P(L_n(T) > C)\\
&\leq \prob\Big(\sup_{t \in [0, T- \delta]}|Z_n^1(t+\delta) - Z_n^1(t)| > \e\Big) + \frac{|v| + K\sup_n\ex(L_n(T))}{C}.
\end{align*}
By Lemma \ref{ch3:lemma:functionalLLN} and Lemma \ref{ch3:LemmaCTight} we know
\[
\lim_{\d \to 0}\limsup_{n\to \infty}\prob\Big(\sup_{t \in [0, T- \delta]}|Z_n^1(t+\delta) - Z_n^1(t)| > \e\Big) = 0.
\]
By this and the uniform moment bound of $\sup_n\ex(L_n(T))$ in Lemma \ref{lemma:L_nBound}, there exists a constant $A$ independent of $C, \e, n$ such that
\[
\lim_{\d \to  0}\limsup_{n \to \infty}\prob\Big(\sup_{t \in [0, T -\delta]}\sup_{t < u,v < t + \delta}|Z_n(u) - Z_n(v)| > \e\Big)
\leq \frac{A}{C}.
\]
By choosing $C$ arbitrarily large, we see
\[
\lim_{\d \to  0}\limsup_{n \to \infty}\prob\Big(\sup_{t \in [0, T -\delta]}\sup_{t < u,v < t + \delta}|Z_n(u) - Z_n(v)| > \e\Big) = 0.
\]
Hence $Z_n$ satisfies (ii) of Lemma \ref{ch3:LemmaCTight}. Condition (i) follows from Remark \ref{ch3:remark:Ctight} and the fact that $Z_n(0) = 0$ almost surely. This completes our verification of conditions (i)-(ii) of Lemma \ref{ch3:LemmaCTight} sufficient for $C$-tightness of $\{Z_n : n \in \N\}$.
\end{proof}

\begin{cor}
The collection of processes $\{(S_n, Z_n, V_n) : n \in \N\}$ is $C$-tight in $D([0, T], \R^3)$ with the Skorohod topology.
\end{cor}
\begin{remark}
As topological spaces, $D([0, T], \R^d)$ with the Skorohod topology is not equivalent to $D([0, T], \R)^d$ with the product topology. However, because the marginals are $C$-tight this a non-issue essentially because uniform convergence to a continuous function becomes the same in both spaces. See the comment in Jacod and Shiryaev \cite[VI.1.21]{Shiryaev}.
\end{remark}
\begin{proof}
Note that $S_n$ is $C$-tight by Donsker's theorem, while $C$-tightness of $V_n = v - KL_n$, and $Z_n$, follow from Corollary \ref{ch3:corL_ntight},
and Lemma \ref{ch3:Z_nIsTight}, respectively. By Skorohod's representation theorem, every subsequence of $(S_n, Z_n, V_n)$ converging to a limiting process $(S, Z, V)$ can be assumed to converge almost surely in the product metric $d \times d \times d$, the product metric on $D([0, T], \R)^3$. By $C$-tightness of the marginals, $S, Z, V$ are all continuous, so that $(S, Z, V)$ is a continuous process. As in Remark \ref{ch3:remark:convInUnif}, this implies the subsequence of $(S_n, Z_n, V_n)$ converges to $(S, Z, V)$ almost surely in the uniform norm, which implies almost sure convergence in $D([0, T], \R^3)$ under the Skorohod metric. Thus, $(S_n, Z_n, V_n)$ is $C$-tight as a collection of processes with paths in $D([0, T], \R^3)$.
\end{proof}

\subsection{Lemma \ref{lemma:sub}: Characterization of subsequential limits}\label{sec:sub}
We prove items (i)-(iv) in Lemma \ref{lemma:sub} separately. The proof of (iii) was
inspired by the proof of L\'evy's theorem given in \cite[Chapter 6]{MortersPeres}, where the authors
essentially note the equivalence of the processes $(U_n + M_n, V_n)$ and $(X_n, V_n^X)$, which we described in 
subsection 4.1, for the case $K = 0.$ L\'evy's theorem is the statement that $(L, |B|)$ and $(M^B, B + M^B)$ yield the same distribution on $C([0, T], \R^2).$ Here $L$ is the local time of $|B|$ at zero. See \cite[Chapter 3.6]{KaratzasShreve} for a detailed statement.
\begin{proof}[{Proof of (i)}]
This follows trivially from the definition of $U_n = S_n + Z_n.$
\end{proof}
\begin{proof}[{Proof of (ii)}]
Recall $S_n$ is a continuous time scaled simple random walk. Since $S_{n_k}$ converges to $S$, $S$
is a Brownian motion by Donsker's theorem.
\end{proof}
We give a brief heuristic for the proof of (iii).
Recall $L_n = 2^{n}\cdot Leb(0 < s < t : U_n(s) = -M_n(s))$ and $V_n = -v + KL_n.$ Each time $M_{n}$ 
increases, $U_{n}$ will make approximately a Geometric($1/2$) number of visits to this new minimum value before $M_n$ increases again. Also, $U_{n}$ will spend approximately an Exp($2^{2n}$) amount of time at each one of these visits. Therefore, $L_n$, which is 
the total amount of time $U_n$ spends on $-M_n$ scaled by $2^n$, is approximately 
\[
\sum_{i = 1}^{2^nM_n}2^n\Exp(2^{2n}) = \sum_{i = 1}^{2^nM_n}\Exp(2^{n}),
\]
where $\Exp(2^n)$ indicates independent exponential random variables of rate $2^n$. If you suppose this sum is
concentrated around its expectation conditional  on $M_n$, then
\[
L_n(t) \approx \sum_{i = 1}^{2^nM_n(t)}2^{-n} = M_n(t).
\]
Furthermore, if $M_n$ converges almost surely to a process $M$ in the uniform norm, one expects $L_n$ to converge to $M$ as well.
\begin{proof}[{Proof of (iii)}]
By tightness, and without loss of generality, assume $L_{n_k} \lra L$ and 
$U_{n_k} \lra U$ almost surely in the uniform norm of continuous 
functions. See Remark \ref{ch3:remark:convInUnif}.
We will use a localization argument by stopping $L_n$ after it reaches a large value. For a positive constant $C > |v|$, define
 \[
 T_C^{(n_k)} = \inf\{t > 0: L_{n_k}(t) > C\}.
 \] 
For each $n_k,$
consider a modification of $(S_{n_k}, Z_{n_k}, U_{n_k})$, denoted $(S_{n_k}^C, Z_{n_k}^C, U_{n_k}^C),$ solving the system (i)-(iii) in subsection 4.1 but replacing (iii) with 
\begin{align*}
&V_n^C(t) = -v + K[2^{n}Leb(0 < s < t : U_n^C(s) = -M_n^C(s)) \land C],\\
&U_n^C = S_n^C + Z_n^C,\\
&M_n^C(t) := M^{U_n^C}(t).
\end{align*}
In other words we are stopping $L_n$ when it reaches $C$ while keeping the other dynamics of the system the same.
Therefore $(S_{n_k}^C, Z_{n_k}^C, U_{n_k}^C)$ is equal to $(S_{n_k}, Z_{n_k}, U_{n_k})$ on the interval $[0, T_C^{(n_k)}]$, a.s., while on $[T_C^{(n_k)}, \infty)$ the process $L_{n_k}^C$ is constantly $C$, and $U_{n_k}$ is the sum of a (scaled) continuous time random walk and an independent jump process of rate $-v + KC2^{n_k}$ and jump size $2^{-n}$.  Fix $m \in 2^{-n_k}\N$. 
We bound the number of positive excursions of $U_{n_k}^C$ above $-m$ conditional on $-m$ being the running minimum. Let 
 \begin{align*}
\tau_{m, i + 1} = \inf\{t > \tau_{m, i} : U_{n_k}^C(t) = -m = -M_{n_k}^C\} \land T,
 \end{align*}
 be the consecutive times $U_{n_k}^C$ visits $-m$ when $m$ is the current value of $M_{n_k}^C$, up until time $T$. (Set $\tau_{m, 0} = 0$). That is, the consecutive times
 $U_{n_k}^C$ visits its running minimum at $-m.$ Let 
  \[
  z_{m,j}^+ = \inf\{ s > 0 : U_{n_k}^C(\tau_{m, j}^{(n_k)} + s)  > U_{n_k}^C(\tau_{m, j}^{(n_k)})\}
  \] and 
  \[
z_{m,j}^- = \inf\{ s > 0 : U_{n_k}^C(\tau_{m, j}^{(n_k)} + s)  < U_{n_k}^C(\tau_{m, j}^{(n_k)})\}
  \] 
denote the amount of time until the next positive and negative, respectively, jump of $U_{n_k}^C$ after the $j^{th}$ excursion starting at $-m$. Then $z_{m,j}^+ \land z_{m,j}^-$ is the time $U_{n_k}^C$ spends on $-m$ during its
 $j^{th}$ visit to its running minimum $-m$.
Since $|L_n^C| \leq C$, the stochastic intensity of $Z_{n_k}^C$ is $2^{n_k}|V_{n_k}^C|$, which is bounded below by $2^{2n_k} - v$ and above by $2^{2n_k} + C2^{n_k}$. We set $v = 0$ in the remaining computations for convenience. Therefore the positive jump times of $U_{n_k}^C = S_{n_k}^C + Z_{n_k}^C$ have intensity $2^{2n_k} + 2^{n_k}|V_{n_k}^C|$ and the negative jump times arrive with intensity $2^{2n_k}$. In other words, $z_{m,j}^- \dist \Exp(2^{2n_k}).$

By Lemma \ref{ch3:jumpRateStochDom} we assume the probability space contains two independent sequences of i.i.d.\ exponential random variables $\{\nu_{i, j}\}_{i,j\in \N}, \ \{e_{i,j}\}_{i,j\in \N}$ that are also independent of $z_{m,j}^-$ and of the position $m$
, with rates $2^{2n_k}$ and $C2^{n_k}$, respectively, and where 
\[
\nu_{m,j} \land e_{m,j} \leq z_{m,j}^+ \leq \nu_{m,j},
\] almost surely. Consequently,
\begin{align}\label{eq:rate_coupling}
z_{m,j}^- \land \nu_{m,j} \land e_{m, j} \leq z_{m,j}^- \land z_{m,j}^+ \leq z_{m,j}^- \land \nu_{m,j}, \text{ almost surely.}
\end{align}
Define
\[
A_j^m := 1_{\{z_{m,j}^- \land \nu_{m,j} = \nu_{m,j}\}}, \ B_j^m := 1_{\{z_{m,j}^- \land z_{m,j}^+ = z_{m,j}^+\}}, \ C_j^m := 1_{\{z_{m,j}^- \land \nu_{m,j} \land e_{m,j} = \nu_{m,j} \land e_{m,j} \}}.
\]
Note that $B_j^m$ is the indicator for whether $U_{n_k}^C$ jumped in the positive direction during its $j^{th}$ visit to its running minimum $-m.$
By construction $A_j^m, B_j^m, C_j^m$ are Bernoulli random variables and are coupled so that 
\begin{align}\label{eq:jump_sandwich}
A_j^m \leq B_j^m \leq C_j^m, \text{ almost surely.}
\end{align}
The definition of $A_j^m, B_j^m, C_j^m$ depend on $n_k$, which is hidden from notation.
While the sequence $(B_j^m: j \geq 1)$ is not an i.i.d.\ sequence, and is not a sequence of independent random variables
since the jump rate changes with time, both $(A_j^m: j \geq 1)$ and $(C_j^m: j \geq 1)$ are i.i.d.\ sequences of 
Bernoulli$(1/2)$, Bernoulli$([1 + C2^{-n_k}]/[2 + C2^{-n_k}]),$ respectively. 
%We will describe the distribution of the first index $j$ such that $A_j^m$ (resp. $B_j^m$) is zero. 
For each $i \in \N$,
denote $Q_i$ as the number of visits to $-i2^{-n_k}$ by $U_{n_k}$ while $M_{n_k}^C = i2^{-n_k}.$ This is the number of visits $U_{n_k}^C$ makes to $-i2^{-n_k}$ when $-i2^{-n_k}$ is the running minimum. We will use \eqref{eq:jump_sandwich} to sandwich $Q_i$ above and below by geometric random variables. Denote $m_i = i2^{-n_k}$. Then
  \[
Q_i = \inf\{ j \geq 1 : B_j^{m_i} = 0 \}.
  \]
That is, $Q_i$ is the number of visits $U_{n_k}^C$ makes to its running minimum $-m_i$ because once a negative jump occurs, i.e. $B_j^{m_i}$ reaches 0, the running minimum decreases to $-(i + 1)2^{-n_k}$.
Consider 
\begin{align*}
W_i &= \inf\{j \geq 1 : A_j^{m_i} = 0\},\\
V_i &= \inf\{j \geq 1 : C_j^{m_i} = 0\}.
\end{align*}
Because $(A_j^{m_i}), (C_j^{m_i})$ are each i.i.d.\ sequences of Bernoulli random variables, $W_i, V_i$
are geometric random variables, and
 \begin{align*}
 &W_i \leq Q_i \leq V_i, \text{ almost surely,}\\
 &\prob(W_i = k) = \Big(\frac{1}{2}\Big)^{k},\\
 &\prob(V_i = k) = \Big(\frac{1}{2 + C/2^{n_k}}\Big)\Big(\frac{1 + C/2^{n_k}}{2 + C/2^{n_k}}\Big)^{k-1}.
 \end{align*}
 That is, $W_i, V_i$ are geometrically distributed with parameters 1/2, $(1 + C2^{-n_k})/(2 + C2^{-n_k})$ respectively. 

% Notice that
% $M^C_{n_k}$ visits all sites between 0 and $m$ up until time $\tau_{m, 1},$ the first time $U_{n_k}^C$ reaches $-m.$ (Recall that $S_{n}$ is identical to $S_n^C$ since it is a continuous time random walk not depending on $L_n$). 
Now that we have sandwiched the number of steps $U_n^C$ makes at a certain level of its running minimum, we will analyze the Lebesgue time the process spends at its running minimum. Since the size of each step is $2^{-n_k}$, 
$M^C_{n_k}$ has visited between $2^{n_k}m- 1$ and $2^{n_k}m + 1$ sites up until time $\tau_{m, 1}$. 

% The process $U_n^C$ may jump up or down after time $\tau_m^{(n_k)}$By Lemma \ref{ch3:jumpRateStochDom} we m the amount of time the process $U_n^C$ spends at $-i2^{n_k}m$  spend
% a certain amount  As $M^C_{n_k}$ is nondecreasing these are the only sites it visits. Let $(z_j, G_j), j \geq 1$ be the joint random variables such that $z_j$ repre
 % For each 
 % $0 < m_i < m$
 %  denote $Q_i$ as the number of visits to $-m_i$ by $U_{n_k}$ while $M_{n_k}^C = m_i.$ This is the number of visits $U_{n_k}^C$ makes to $-m_i$ when $-m_i$ is the running minimum. We will now sandwich $Q_i$ above and below geometric random variables.
 %  Note
 % $V_{n_k}^C$ is bounded by $C2^{n_k}$ by definition, and at each of these visits $U_{n_k}$ will spend at least exponential time of rate $2^{2n_k+1} + C2^{n_k}$ (sum
 %  the rates of moving up and down) and at most an exponential time of rate
 %  $2^{2n_k+1}$. By Lemma \ref{ch3:jumpRateStochDom} we may enlarge the probability space to support geometric random variables $W_i, V_i$ such 
 %  that
 % \begin{align*}
 % &W_i \leq Q_i \leq V_i, \text{ almost surely,}\\
 % &\prob(W_i = k) = \Big(\frac{1}{2}\Big)^{k},\\
 % &\prob(V_i = k) = \Big(\frac{1}{2 + C/2^{n_k}}\Big)\Big(\frac{1 + C/2^{n_k}}{2 + C/2^{n_k}}\Big)^{k-1}.
 % \end{align*}
 % That is, $W_i, V_i$ are geometrically distributed with parameters 1/2, $1/(2 + C2^{-n_k})$ respectively. By Lemma \ref{ch3:lemma:minOfExpsIndofIndicator} $W_i, V_i$ are independent of the exponential random variables determining the time spent at each particular level. 
 Let $T_i$ be the time that $M_{n_k}^C$ spends at the site $m_i$, for $0 \leq i2^{-n_k} =:m_i < M_{n_k}^C(s)$ and a given $s \in [0, T].$
  % Then
 % \begin{equation}
 % \ds \frac{2}{2^{n_k + 1} + C} \leq 2^{n_k}\ex T_i \leq \frac{2 + C2^{-n_k}}
 % {2^{n_k+1}}.
 % \label{eq:b1}
 % \end{equation}
% Since $|L_n^C| \leq C$, the stochastic intensity of $Z_{n_k}^C$ is $2^{n_k}|V_{n_k}^C|$, which is bounded below by $2^{2n_k} - v$ and above by $2^{2n_k} + C2^{n_k}$. We set $v = 0$ in the remaining computations for convenience, so the jump times of $U_{n_k}^C = S_{n_k}^C + Z_{n_k}^C$ have intensity $2^{2n_k} + 2^{n_k}|V_{n_k}^C|$, which is bounded below by $2^{2n_k}$ and above by $2^{2n_k} + C2^{n_k}.$ 
By definition of $\nu_{m,j}, e_{m,j}, z_{m,j}$ and the inequality \eqref{eq:rate_coupling}, for $0 \leq i2^{-n_k} < M_{n_k}^C(s)$ we have
\begin{align}
\begin{split}\label{eq:Phi'_T_Phi_Ineq}
&\wt{\Phi}^{i}_{n_k} :=
\sum_{j = 1}^{W_i}z_{m_i, j}^- \land \nu_{m_i, j} \land e_{m_i, j}
\leq T_i \leq
\sum_{j=1}^{W_i}z_{m_i, j}^- \land \nu_{m_i, j} + \sum_{j = W_i}^{V_i}z_{m_i, j}^- \land \nu_{m_i, j}\\
&=: \Phi_{n_k}^i + \sum_{j = W_i}^{V_i}z_{m_i, j}^- \land \nu_{m_i, j},
\end{split}
\end{align}
almost surely. This is because $U_n^C$ spends at least $W_i$ steps at the running minimum $-m_i$, each step spending at least $z_{m_i, j}^- \land \nu_{m_i, j} \land e_{m_i, j}$ Lebesgue amount of time for each $1 \leq j \leq W_i,$ giving the lower bound. The upper bound is the same reasoning. By Lemma \ref{ch3:lemma:minOfExpsIndofIndicator}, $z_{m_i,j}^- \land \nu_{m_i,j}$ is an \Exp$(2^{2n +1})$ random variable independent from $A_j^{m_i}$. Because $z_{m_i, j}^- \land \nu_{m_i, j} \land e_{m_i, j}$ is a measurable function of
$z_{m_i, j}^- \land \nu_{m_i, j}$ and $e_{m_i, j}$, both of which are independent from $A_j^{m_i},$ $z_{m_i,j}^- \land \nu_{m_i,j} \land e_{m_i,j}$ is independent from $A_j^{m_i}$ as well. Consequently $z_{m_i, j}^- \land \nu_{m_i, j} \land e_{m_i, j}$, for $1\leq j \leq W_i$, are independent from $W_i.$ Similarly $V_i$ is independent of $z_{m_i, j}^- \land \nu_{m_i, j}$ for $1 \leq j \leq V_i.$ Since by definition,
\[
L_{n_k}^C(s) = \sum_{i = 1}^{2^{n_k}M_{n_k}^C(s)}2^{n_k}T_i,
\]
\eqref{eq:Phi'_T_Phi_Ineq} implies we can sandwich $L_{n_k}^C(s)$ by summing these upper bounds and lower bounds of times spent at each intermediate level. That is,
\begin{align}
\begin{split}
&\sum_{i = 1}^{2^{n_k}M_{n_k}^C(s) - 1}2^{n_k}\wt{\Phi}^{i}_{n_k} \leq L_{n_k}^C(s)\\ &\leq \sum_{i = 1}^{2^{n_k}M_{n_k}^C(s) + 1} 2^{n_k}\Phi^i_{n_k} + \sum_{i = 1}^{2^{n_k}M_{n_k}^C(s) + 1} \sum_{j = W_i}^{V_i}z_{m_i, j}^- \land \nu_{m_i, j}\\
&=: R_{n_k}(s),
\label{eq:sumLevelCounts}
\end{split}
\end{align}
and this inequality holds for all $s \in [0, T]$, almost surely.

Now we will apply the squeeze theorem to \eqref{eq:sumLevelCounts} and show the left hand and right hand of that inquality, and hence $L_{n_k}$, converge to $M^U$ on $[0, T]$ in probability, hence for some subsequence of $n_k$ the convergence holds almost surely.  The sum of a Geometric($p$) number of independent exponentials of rate $\ld$ is exponential with rate $p\ld$ provided the number of exponential random variables being summed is independent of the exponential random variables themselves. Therefore, since $W_i$ (resp. $V_i$) is geometric and independent of $z_{m_i, j}^- \land \nu_{m_i, j} \land e_{m_i, j}$ for $1 \leq j \leq W_i$ (resp. $z_{m_i, j}^- \land \nu_{m_i, j}$ for $1 \leq j \leq V_i$), $\wt{\Phi}^{i}_{n_k}$ is 
distributed as an exponential of rate $(2^{2n_k+1} + C2^{n_k})/2$. Similarly $\Phi^i_{n_k}$ has exponential rate $2^{2n_k}.$
We think of $2^{n_k}\wt{\Phi}^{i}_{n_k}$ as an exponential random variable with rate approximately
$2^{n_k},$ while in fact $2^{n_k}\Phi^i_{n_k}$ is an exponential with rate exactly $2^{2n_k}.$ Because 
$\wt{\Phi}_{n_k}^i$ is measurable with respect to $\mc{F}_{(\tau_{m_i, 1}, \tau_{m_{i + 1}, 1}]}$ and independent from $\mc{F}_{\tau_{m_i, 1}}$,
$(\wt{\Phi}_{n_k}^i : 1 \leq i \leq W_i)$ is a collection of independent exponential random variables by the strong Markov property. In other words, $\wt{\Phi}_{n_k}^i$ depends only on the excursion between these two hitting times and does not depend on the initial position of these excursions.
Define
\[
T_C := \inf\{t> 0 : L(t) > C \}
\]
where $L = \lim_{n_k}L_{n_k}$. For any $0<\e<C$, it is clear that $\liminf_{n_k}T_{C}^{(n_k)} \geq T_{C -  \e}$, and as a result 
\[
L_{n_k}^C  = L_{n_k}, U_{n_k}^C = U_{n_k}
\] on $[0,  T \land T_{C - \e}]$ for large enough $n_k$, almost surely.
Because $U_{n_k}^C(\cdot)$ converges uniformly on $[0, T]$ to the continuous process $U$, almost surely, we know
 $M_{n_k}^C(\cdot)$ converges uniformly on $[0,  T \land T_{C - \e}]$ to $M^{U}$.
We will show that the left and right hand sides of \eqref{eq:sumLevelCounts} converge almost surely to $M^{U}(s\land T_{C-\e})$ for each fixed $s \in [0, T]$. We go through the details for the left hand side, and the right hand is similar. For ease of notation we denote  $s' = s\land T_{C-\e}.$
Note
\begin{align}
\begin{split}\label{ch3:DifferenceSums}
&\left|\sum_{i = 1}^{2^{n_k}M_{n_k}^C(s')}2^{n_k}\wt{\Phi}^{i}_{n_k} - \sum_{i = 1}^{2^{n_k}M^{U^C}(s')}2^{n_k}\wt{\Phi}^{i}_{n_k}\right|\\
&\leq \sum_{i = 2^{n_k}[M_{n_k}^C(s') \land M^{U}(s')]}^{2^{n_k}[M_{n_k}^C(s') \lor M^{U}(s')]}2^{n_k}\wt{\Phi}^{i}_{n_k} \leq \frac{1}{2^{n_k}}\sum_{i = 1}^{2^{n_k}|M_{n_k}^C(s') - M^{U}(s')|}e_i, 
\end{split}
\end{align}
almost surely, where $e_i$ are i.i.d.\ Exp$(1)$, and that are independent from $M_{n_k}^C$. The last inequality comes from Lemma \ref{ch3:jumpRateStochDom} and the fact that $2^{n_k}\wt{\Phi}^{i}_{n_k} \dist \exp(2^{n_k} + C/2)$ are i.i.d.\ Since
$|M_{n_k}^C(s') - M^{U}(s')| \to 0$, almost surely, the strong law of large numbers implies that 
\[
\frac{1}{2^{n_k}}\sum_{i = 1}^{2^{n_k}|M_{n_k}^C(s') - M^{U}(s')|}e_i \lra 0, \text{ almost surely}.
\]
We can express $2^{n_k}\wt{\Phi}^{i}_{n_k}$ as $2^{-n_k}u_i^k$ where $u_i \dist \exp(1 + C2^{-(n_k + 1)})$ are i.i.d., so
\[
\sum_{i = 1}^{2^{n_k}M^{U}(s')}2^{n_k}\wt{\Phi}^{i}_{n_k} = \frac{1}{2^{n_k}}\sum_{i = 1}^{2^{n_k}M^{U}(s')}u_i^k.
\]
We condition on $M^{U}(s')$ to compute
\begin{align*}
\Var\left( \frac{1}{2^{n_k}}\sum_{i = 1}^{2^{n_k}M^{U}(s')}u_i^k \Big| M^{U}(s')\right) &= \frac{1}{2^{2n_k}}2^{n_k}M^{U}(s')\frac{1}{(1 + C2^{-(n_k + 1)})^2}\\
 &= \frac{M^{U}(s')}{2^{n_k} + C + C^22^{-(n_k + 2)}},
\end{align*}
which approaches zero. The expectation conditional on $M^{U}(s')$ is
\[
\ex\left( \frac{1}{2^{n_k}}\sum_{i = 1}^{2^{n_k}M^{U}(s')}u_i^k \Big| M^{U}(s')\right) = M^{U}(s')\frac{1}{1 + C2^{-(n_k + 1)}}.
\]
Consequently, 
\[
\sum_{i = 1}^{2^{n_k}M^{U}(s')}2^{n_k}\wt{\Phi}^{i}_{n_k} = \frac{1}{2^{n_k}}\sum_{i = 1}^{2^{n_k}M^{U}(s')}u_i^k \lra M^{U}(s'),
\]
in probability, and by \eqref{ch3:DifferenceSums}, 
\begin{align}
\sum_{i = 1}^{2^{n_k}M_{n_k}^C(s')}2^{n_k}\wt{\Phi}^{i}_{n_k} \lra M^{U}(s'), \label{ch3:leftConv}
\end{align}
in probability. Similarly for the right side of \eqref{eq:sumLevelCounts}, one can show 
\begin{align}
\sum_{i = 1}^{2^{n_k}M_{n_k}^C(s')}2^{n_k}\Phi^{i}_{n_k} \lra M^{U}(s'),\label{ch3:rightConv}
\end{align}
in probability. Now the term $R_{n_k}(s')$ converges to $M^{U}(s')$ once we demonstrate
\[
\sum_{i = 1}^{2^{n_k}M_{n_k}^C(s') + 1}\sum_{j = W_i}^{V_i}z_{m_i, j}^{-} \land \nu_{m_i, j}
\] 
converges to zero in probability. But this follows from two applications of Wald's lemma, the second of which
uses the filtration generated (for fixed $i$) by $z_{m_i, j}^-, \nu_{m_i, j}$ and $e_{m_i, j}$ to compute
\begin{align*}
\ex\left(\sum_{j = W_i}^{V_i}z_{m_i, j}^{-} \land \nu_{m_i, j}\right) &= \left(\ex(V_i - W_i) + 1\right)\ex(z_{m_i, j}^- \land \nu_{m_i, j}) \\
&= \left(2 - \frac{2 + C2^{-n_k}}{1 + C2^{-n_k}}\right)\cdot \frac{1}{2^{2n_k + 1} + C2^{n_k}} =: a_{n_k},
\end{align*}
which clearly approaches zero. Using Lemma \ref{ch3:momentBoundMn}, the moment bound hypothesis of Wald's equation is satisfied since we have $\ex(M_n^C(s'))\leq\ex(M_n^C(T)) \leq 2(2T + T|v|(2T)^{1/2} + 2|v|T)^{1/2}$, a 
uniform bound with respect to $n$. Thus, the second application of Wald's lemma gives
\begin{align*}
\ex\left(\sum_{i = 1}^{2^{n_k}M_{n_k}^C(s') + 1}\sum_{j = W_i}^{V_i}z_{m_i, j}^{-} \land \nu_{m_i, j}\right) &=
\ex(2^{n_k}M_{n_k}^C(s') + 1)\cdot a_{n_k},
\end{align*}
which approaches zero as $n_k \to \infty$. Consequently,
\[
\sum_{i = 1}^{2^{n_k}M_{n_k}^C(s') + 1}\sum_{j = W_i}^{V_i}z_{m_i, j}^{-} \land \nu_{m_i, j}
\] does indeed converge to zero in probability, and therefore 
$R_{n_k}(s')$ converges to $M^{U^C}(s')$ in probability.

Because convergence in probability implies almost sure convergence for some subsequence, we can find a common subsequence $n_k'$ where both \eqref{ch3:leftConv} and \eqref{ch3:rightConv} occur almost surely, for fixed $s$, where $s' = s \land T_{C-\e}$. We relabel $n_k'$ as $n_k$. Similarly we can use a cantor diagonalization to find a further subsequence where \eqref{ch3:leftConv} and \eqref{ch3:rightConv} occur for all rationals in $[0, T\land T_{C-\e}]$, almost surely.
Applying the squeeze theorem to the inequality \eqref{eq:sumLevelCounts} then yields
\begin{align}\label{eq:squeeze}
0 = \lim_{n_k \to \infty}|L_{n_k}^C(s') - M^{U}(s')| = \lim_{n_k \to \infty}|L_{n_k}(s') - M^{U}(s')| =  |L(s') - M^U(s')|
\end{align}
where $s' = s\land T_{C-\e}$, for any $s \in \mathbb{Q} \cap [0, T]$, almost surely.
Therefore $L = \lim_{n_k}L_{n_k}  = \lim_{n_k}L_{n_k}^C = M^{U^C} = M^U$ for all rational numbers in $[0, T_{C-\e} \land T] \subset [0, \liminf_{n_k}T_{n_k}^{C} \land T],$ almost surely. So $L = M^U$ on $[0, T_{C - \e} \land T]$, almost surely, since both processes are continuous.

Letting $C$ approach infinity, $\prob(T_C \geq T) \lra 1$, since $L$ is a finite process, which yields $M^U = L$ on $[0, T],$ almost surely, completing the proof of (iii).
\end{proof}

\begin{proof}[Proof of (iv)]
As in the other proofs, Remark \ref{ch3:remark:convInUnif} and Lemma \ref{ch3:lemma:tightness} allow us to assume without loss of generality that for the subsequence $(S_{n_m}, Z_{n_m}, V_{n_m}, L_{n_m})$,
\begin{equation}
(S_{n_m}, Z_{n_m}, V_{n_m}, L_{n_m}) \lra (S, Z, V, L)
\label{eq:assum}
\end{equation}
almost surely, in the uniform norm on $C([0, T], \R);$ we have $U_{n_m} \to U$
on $C([0, T], \R)$ as well. In the previous proof of (iii) we showed $L(t) = M^U(t)$ for each $t \in [0, T]$, almost surely. In this proof we wish to show
\begin{equation}
Z(t) = \int\limits_0^tV(x)\, \text{d}x, \ \text{for } t\in[0, T],
\label{eq:eqZ}
\end{equation}
almost surely, where $V = KM^U - v.$ We take $v \leq 0$ for the time being and reduce
to this case at the end. It suffices to demonstrate that for each $s \in [0, T]$ there is a subsequence $n_m'$
such that
\begin{align}
Z_{n_m'}(s) {\longrightarrow} \int\limits_0^s V(x)\,\text{d}x \label{eq:limitZ}, \text{ almost surely.}
\end{align}
By a Cantor diagonalization $Z(\cdot)$ and 
$\ds \int_0^\cdot M^U(s)\,\text{d}s$ will agree for all rationals in $[0, T],$ almost surely. The two processes will then agree on $[0, T]$, almost surely,
because both processes are continuous.
For a given $n,$ 
\[
\widetilde{Z}_n(s) := 2^nZ_n(s)
\]
counts the
number of jumps of $Z_n$ by time $s.$ Equivalently, this counts the
number of arrival times $\{u_k : k \geq 1\}$ of jumps
by the process $Z_n.$ (We hide the dependence of $u_k$ on $n$ for convenience). 
For $C > 0$, let
\begin{align*}
&\tau_C^{(n)} = \inf\{ s > 0 : V_n(s) > C\},\\
&\overline{\alpha}_k = \sup_{s \in {[u_k, u_{k+1}]}}V_n(s),\\
&\underline{\alpha}_k = \inf_{s \in {[u_k, u_{k+1}]}}V_n(s).
\end{align*}
% Our assumptions guarantee
% \[
% \limsup \tau_C^{(n_m)} \land T \overset{a.s.}{\leq} \tau_C \land T = \inf\{s > 0 : M^U > C\} \land T. \label{eq9'}
% \]
Assume for the time being that for every fixed $\delta > 0$,
\begin{align}
\sup\{(u_{i+1} - u_i)  : \wt{Z}_{n_m}(\delta \land s) \leq i \leq \widetilde{Z}_{n_m}(s)\}
\lra 0, \text{ in probability.}
\label{eq:net}
\end{align}
Then there exists a subsequence $n_m'$, which we relabel as $n_m$, such that $\sup\{(u_{i+1} - u_i)  : \wt{Z}_{n_m}(\delta \land s) \leq i \leq \widetilde{Z}_{n_m}(s)\} \lra 0,$ almost surely. 
We use the time between jumps, $u_{k+1} - u_k$, as the time step in a Riemann sum
approximation of the integral in \eqref{eq:limitZ}. 
By the definition
of $\overline{\alpha}_k, \underline{\alpha_k}$ and the exponential representation of the gap times given in \eqref{eq:Exponential_Rep_JumpTimes}, there is a 
sequence $\mu_k$ of i.i.d. $\Exp(2^n)$ random variables such that

$$\underline{\alpha}_k(u_{k+1} - u_k) \leq \mu_k \leq
\overline{\alpha}_k(u_{k+1} - u_k),$$
almost surely. Therefore,
\begin{equation}
\sum_{k = \wt{Z}_{n_m}(\delta\land s)}^{\wt{Z}_{n_m}(s)}\underline{\alpha}_k(u_{k+1} - u_k)
\leq \sum_{k = \wt{Z}_{n_m}(\delta \land s)}^{\wt{Z}_{n_m}(s)}\mu_k
\leq \sum_{k = \wt{Z}_{n_m}(\delta \land s)}^{\wt{Z}_{n_m}(s)}\overline{\alpha}_k(u_{k+1} - u_k)
\label{eq:bound}
\end{equation}
where we define the left and right sums to be zero should the set of such 
indices $\wt{Z}_{n_m}(\delta \land s) \leq k \leq \wt{Z}_{n_m}(s)$ be empty.

From \eqref{eq:net} together with \eqref{eq:assum} and Riemann integrability
of the limiting function $V$, 
\begin{equation}
\lim_{n_m' \to \infty}\sum\limits_{k = \wt{Z}_{n_m}(\delta\land s)}^{\wt{Z}_{n_m}(s)}\underline{\alpha}_k(u_{k+1} - u_k) 
= \int\limits_{\delta \land s}^sV(x)\, \mathrm{d}x 
= \lim_{n_m' \to \infty}\sum_{k = \wt{Z}_{n_m}(\delta \land s)}^{\wt{Z}_{n_m}(s)}\overline{\alpha}_k(u_{k+1} - u_k),
\label{eq:squeeze}
\end{equation}
where convergence holds uniformly on $[0, T]$, almost surely. By the squeeze theorem,
\begin{equation}
\sum_{k = \wt{Z}_{n_m}(\delta \land s)}^{\wt{Z}_{n_m}(s)}\mu_k
\lra \int\limits_{\delta \land s}^sV(x)\, \mathrm{d}x,
\end{equation}
almost surely, as well.
\noindent
Since the $\mu_k$ are i.i.d.\ exponential r.v.'s of rate $2^{n_m}$ and 
$\wt{Z}_{n_m}(s) = 2^{n_m}Z_{n_m}(s)$ with $Z_{n_m}(\cdot) \to Z(\cdot)$
almost surely, the law of large numbers implies
\[
\sum_{k = \wt{Z}_{n_m}(\delta \land s)}^{\wt{Z}_{n_m}(s)}\mu_k \overset{a.s.}{\lra}
Z(s) - Z(\delta \land s),
\]
for each $s \in [0, T].$
Therefore 
\begin{align}
Z(s) - Z(\delta \land s) =\int\limits_{\delta \land s}^sV(x)\, \mathrm{d}x\label{ch3:ivEq1}
\end{align}
for each $s$ in $[0, T]$, almost surely. Since $Z(\delta) \lra 0$ almost surely and $\int_0^\delta V(x)\, \mathrm{d}x \lra 0,$ as $\delta \to 0$, this gives 
\[
Z(s) = \int_0^sV(x)\, \mathrm{d}x
\]
as desired.
% $$
% \sum_{k = 0}^{\wt{Z}_{n_m}(s)}\mu_k = \sum_{k = \wt{Z}_{n_m}(\delta \land s)}^{\wt{Z}_{n_m}(s)}\mu_k + \sum_{k = 0}^{\wt{Z}_{n_m}(\delta \land s)}\mu_k.
% $$

% Since the $\mu_k$ are i.i.d. $\Exp(2^n)$ random variables, $\ds \sum_{k = 0}^{\wt{Z}_{n_m}(s)}\mu_k - 2^{-n}\wt{Z}_{n_m}(s)$ is a martingale, and in fact,
% $\ds \sum_{k = 0}^{\wt{Z}_{n_m}(s)}\mu_k - 2^{-n}\wt{Z}_{n_m}(s) \overset{P}{\to} 0.$
% Because $\tau_\e^{(n_m')}\land s \to \tau_\e \land s$ we see
% $\ds  \sum_{k = 0}^{\wt{Z}_{n_m}(\tau_\e^{(n_m')} \land s)}\mu_k - 2^{-n}\wt{Z}_{n_m}(\tau_\e^{(n_m')} \land s) \overset{P}{\to} 0$ as well.

% Since $Z_n(s) = 2^{-n}\wt{Z}_n(s)$ we have a further subsequence with
% $$\sum_{k = \wt{Z}_{n_m}(\tau_\e^{(n_m')}) \land s}^{\wt{Z}_{n_m}(s)}\mu_k \lra Z(s) - Z(\tau_\e \land s).$$

% Because $L(s) = M^U(s)$ one can use a Girsanov transform to see that
% $L(s) > 0$ for every $s > 0$ which holds under our assumption that $v =0.$
% By continuity $\tau_\e \land s \to 0$ as $\e \to 0,$ and by continuity of $Z$
% $Z(\tau_\e \land \e) \to 0$ almost surely as $\e \to 0.$ For any $\eta, \delta > 0$
% choose $\e$ small enough so that $Z(\tau_\e \land s) < \eta$ with 
% probability at least $1 - \delta.$ Then 
% $$\ds \prob\Big(\limsup_{n_m' \to \infty}|\sum_{k=0}^{\wt{Z}_{n_m}(s)} \mu_k - Z(s)| > \eta\Big) > 1- \delta.$$
% Particularly, this shows that $\ds \sum_{k=0}^{\wt{Z}_{n_m}(s)} \mu_k$
% converges to $Z(s)$ in probability.

To demonstrate \eqref{eq:net},
recall the jump process $Z_{n_k}$ determining the gap between jump times $u_{i+1} - u_i$ has an intensity process $2^{n_k}|V_{n_k}|$ that is bounded below by $\e2^{n_k}$ on the interval $[\tau^{(n_k)}_\e, \infty).$ Heuristically, on this interval the intensity cannot
be too small so the inter-arrival times are not too large. (This is where we use the fact that $v\leq0$, so that $|V_{n_k}| = V_{n_k}$. In the case that $v > 0$ the intensity $2^{n_k}|V_{n_k}|$ will cross zero, which we handle at the end.) By Lemma \ref{ch3:jumpRateStochDom} there exists
an i.i.d.\ sequence $v_i$ of exponential random variables with rate $\e2^{n_m}$ that stochastically dominate $u_{i+1} - u_i.$ 
We have
\begin{align}
\{\tau_\e^{(n_m)} > \d\} = \{V_{n_m}(\d) \leq \e\}.
\label{eq:limsupST}
\end{align}
For $0 <\eta \ll 1, C > 0,$
\begin{align*}
&\prob(\sup\{(u_{i+1} - u_i) : \widetilde{Z}_{n_m}(\delta) \leq i \leq \widetilde{Z}_{n_m}(t)\} > \eta)\\
&\leq \prob(\sup\{(u_{i+1} - u_i) : \widetilde{Z}_{n_m}(\tau_\e^{(n_m)}) \leq i \leq \widetilde{Z}_{n_m}(t)\} > \eta, \tau^{(n_m)}_\e \leq \delta) + \prob(\tau_\e^{(n_m)} > \delta)\\ 
&\leq
\prob(\sup\{v_i : 1 \leq i \leq \widetilde{Z}_{n_m}(t)\} > \eta) + \prob(\tau_\e^{(n_m)} > \delta)\\
&\leq \prob(\sup\{v_i, 1 \leq i \leq C2^{n_m}\} > \eta, \wt{Z}_{n_m}(t) \leq C2^{n_m}) + \prob(\wt{Z}_{n_m}(t) > C2^{n_m}) + \prob(\tau_\e^{(n_m)} > \delta)\\
&\leq \prob(v_i > \eta : \text{ some } 1 \leq i \leq C2^{n_m}) 
+ \prob(Z_{n_m}(t) > C) + \prob(\tau_\e^{(n_m)} > \delta) \\
&\leq C2^{n_m}\prob(v_i > \eta) + \prob(Z_{n_m}(t) > C)+ \prob(\tau_\e^{(n_m)} > \delta)\\
&\leq C2^{n_m}\exp(-\eta\e2^{n_m}) + \prob(Z_{n_m}(t) > C) + \prob(\tau_\e^{(n_m)} > \delta), \\
&= C2^{n_m}\exp(-\eta\e2^{n_m}) + \prob(Z_{n_m}(t) > C) + \prob(V_{n_m}(\d) \leq \e), \text{ by \eqref{eq:limsupST}}.
\end{align*}
Taking $\limsup$ with respect
to $n_m$ on both sides and applying the assumption that $Z_n \to Z$ and $V_n \to V$ almost
surely, we have
\begin{align*}
&\limsup_{n_m\to\infty}\prob(\sup\{(u_{i+1} - u_i) : \widetilde{Z}_{n_m}(\delta) \leq i \leq \wt{Z}_{n
_m}(t)\} > \eta) \leq \prob(Z(t) > C) + \prob(V(\d) \leq \e).
\end{align*}
Since $C, \e > 0$ are arbitrary and $V(0) \geq 0$ on our assumption $v \leq 0$,
\[
\limsup_{n_m\to\infty}\prob(\sup\{(u_{i+1} - u_i) : \widetilde{Z}_{n_m}(\delta)
 \leq i \leq \wt{Z}_{n_m}(t)\} > \eta) = 0,
 \]
for every fixed $\delta > 0$, proving \eqref{eq:net}.

To show the case $v > 0$ reduces to $v = 0$, notice that
\begin{align}\label{eq:bound_T_e}
 T_{-3\e/2} \leq \liminf_{n_m \to \infty}\tau_{-\e}^{(n_m)} \leq \limsup_{n_m \to \infty}\tau_{\e}^{(n_m)} \leq T_{3\e/2},
 \end{align}
 where
 \[
 T_a := \inf\{t>0:V(t)>a\}.
 \]
 For almost each $\omega$ in our probability space there is an $N(\omega)$ such that $|V_{n_m}|$ is 
 monotone and bounded away from zero on the intervals
  $[0, \liminf_{n_m \to \infty} \tau_{-\e}^{(n_m)}]\supset [0, T_{-3\e/2}\land T]$ and $[\limsup_{n_m \to \infty}\tau_{\e}^{(n_m)}, T]\supset [T_{3\e/2}\land T, T]$, for 
 all $n_m \geq N(\omega).$ With this fact and \eqref{eq:bound_T_e} we can apply the proof thus far to show 
 \[
 Z(t)-Z(s) = \int_s^tV(x)\, \mathrm{d}x \text{ for } s, t \in [0, T_{-3\e/2} \land T], \text{ or } s, t \in [T_{3\e/2} \land T, T].
\]
In addition to this, an $L_\infty$ bound gives
\[
\int_{T_{-3\e/2} \land T}^{T_{3\e/2} \land T}|V(x)|\, \mathrm{d}x \leq (3\e/2)T, \text{ almost surely, }
\]
which goes to zero as $\e \to 0.$
It follows that $\ds Z(s) = \int_0^s V(x)\, \mathrm{d}x$ for $s \in [0, T]$ in the case $v < 0$ as well.
\end{proof}

\section{Mutlidimensional Analog of BMID} \label{sec:mutli}
In 2007, White constructed a multidimensional analog, see \cite{white2007}, whose stationary distribution was found by Bass, Burdzy, Chen and Hairer \cite{bass2010stationary}. This multidimensional analog is a pair of processes $(Z, V)$ where
$Z$ is a diffusion reflecting inside a sufficiently smooth domain $D \subset \R^n$, and $V$ is its drift.
This drift is the inward normal integrated against the local time $Z$ spends on $\partial D$. 
That is,
\begin{align}
\begin{split}
Z(t) &= B(t) + \int_0^t\eta(Z(s))\,\mathrm{d}L(s) + \int_0^tV(s)\, \mathrm{d}s, \label{ch3:eq1}\\
V(t) &= V_0 + \int_0^t\eta(Z(s))\, \mathrm{d}L(s),
\end{split}
\end{align}
\noindent
where $\eta(x)$ is the inward unit normal for $x \in \partial D$ and $t \to L(t)$ is a nondecreasing
continuous function flat off of $\partial D.$ By this we mean $L$ increases only on $Z^{-1}(\partial D).$
The authors show $(Z, V)$ has a stationary distribution of $\mu \times \gamma$, where $\mu$ is the uniform
distribution on $D$ and $\gamma$ is the Gaussian distribution on $\R^n$. This is interesting in part because the stationary distribution of the drift is always Gaussian and does not depend on $D$, and also because the stationary distribution is always a product form.
When $Z$ is one dimensional, and $D = [0, \infty)$, the process $Z$ is one dimensional reflected BMID which is the process introduced by Knight. 

\section*{Acknowledgements}
CB is a Zuckerman Postdoctoral scholar at Technion-Israel's Institute of Technology, Industrial Engineering and Management, Haifa, Israel, 32000. 
The preparation of this manuscript was partially supported by FNS 200021\_175728/1. During this research the author was graduate student at the University of Washington and visited Universidad de Chile.

\end{document}

%% file: preamble2.tex
\usepackage{dsfont, pifont,amsfonts, amssymb, amsmath}
\usepackage{amsthm}
\usepackage[usenames,dvipsnames]{pstricks}
\usepackage{graphicx}
\usepackage{relsize}
\usepackage{verbatim}
\usepackage{enumitem}
\usepackage{dsfont}
\usepackage{epstopdf}
\usepackage{mathrsfs}
\usepackage{hyperref}
\usepackage[]{algorithm2e}

\usepackage{enumitem}
\setenumerate[1]{label=\thesection.\arabic*.}
\setenumerate[2]{label*=\arabic*.}

\newcommand{\ex}{\mathbb{E}}
\newcommand{\prob}{\mathbb{P}}

\newcommand{\e}{\epsilon}
\newcommand{\ds}{\displaystyle}
\newcommand{\N}{\mathbb{N}}
\newcommand{\Z}{\mathbb{Z}}
\newcommand{\R}{\mathbb{R}}
\newcommand{\Var}{\mathrm{Var}}
\newcommand{\wt}[1]{\widetilde{#1}}
\newcommand{\lra}{\longrightarrow}

\newcommand{\mc}[1]{\mathcal{#1}}
\newcommand{\ld}{\lambda}
\renewcommand{\d}{\delta}
\newcommand{\F}{\mathcal{F}}

\newcommand{\la}{\lambda}
\newcommand{\dist}{\overset{d}{=}}

\newcommand{\Exp}{Exp}

\newcommand{\sign}{sign}

\usepackage{chngcntr}